\documentclass[11pt]{article}

\usepackage{amsfonts}
\usepackage{amsthm}
\usepackage{amsmath}
\usepackage{amssymb}

\usepackage{epsfig}
\usepackage[percent]{overpic}
\usepackage{graphics}
\usepackage{psfrag}

\usepackage{verbatim}
\usepackage{color}

\usepackage[all]{xy}
\usepackage{makeidx}
\usepackage{enumerate}
\usepackage[letterpaper,margin=0.9in]{geometry}

\let\oldbibliography\thebibliography
\renewcommand{\thebibliography}[1]{%
  \oldbibliography{#1}%
  \setlength{\itemsep}{-1.2mm}%
}

\newcommand{\rmd}{\mathrm{d}}

\renewcommand{\epsilon}{\varepsilon}

\theoremstyle{plain}
\newtheorem{thm}{Theorem}[section]

\newtheorem{theorem}{Theorem}

\newtheorem{cor}[thm]{Corollary}
\newtheorem{lem}[thm]{Lemma}

\theoremstyle{definition}
\newtheorem{defn}[thm]{Definition}

\newtheorem{assms}[thm]{Assumptions}
\newtheorem{remark}[thm]{Remark}
\newtheoremstyle{myremark}
  {3pt}
  {3pt}
  {\small \rmfamily}
  {5pt}
  {\rmfamily}
  {:}
  {.5em}
  {}

\newcommand{\nwc}{\newcommand}
\nwc{\red}[1]{\textcolor{red}{#1}}
\nwc{\blue}[1]{\textcolor{blue}{#1}}
\def\R{\mathbb{R}}

\def\txtc{{\textnormal{c}}}
\def\txtd{{\textnormal{d}}}
\def\txte{{\textnormal{e}}}

\def\ra{\rightarrow}
\def\I{\infty}

\newcommand{\be}{\begin{equation}}
\newcommand{\ee}{\end{equation}}
\newcommand{\benn}{\begin{equation*}}
\newcommand{\eenn}{\end{equation*}}
\newcommand{\bea}{\begin{eqnarray}}
\newcommand{\eea}{\end{eqnarray}}
\newcommand{\beann}{\begin{eqnarray*}}
\newcommand{\eeann}{\end{eqnarray*}}

\newcommand{\myendex}{$\blacklozenge$\end{ex}}
\newcommand{\myendexerc}{$\lozenge$\end{exerc}}
\newcommand{\myendpexerc}{$\lozenge$\end{pexerc}}

\begin{document}
\numberwithin{equation}{section}
\author{Maximilian Engel\thanks{Department of Mathematics, Freie Universit{\"a}t Berlin, Arnimallee 6, 14195 Berlin, Germany.}, Marios-Antonios Gkogkas\thanks{
Faculty of Mathematics, Technical University 
of Munich, 85748 Garching b.~M\"unchen, Germany}, Christian Kuehn\footnotemark[2]}

\title{Homogenization of Coupled Fast-Slow Systems \\ 
via Intermediate Stochastic Regularization}

\maketitle

\begin{abstract}
In this paper we study coupled fast-slow ordinary differential equations (ODEs) with small time scale separation parameter $\epsilon$ such that, for every fixed value of the slow variable, the fast dynamics are sufficiently chaotic with ergodic invariant measure. Convergence of the slow process to the solution of a homogenized stochastic differential equation (SDE) in the limit $\epsilon$ to zero, with explicit formulas for drift and diffusion coefficients, has so far only been obtained for the case that the fast dynamics evolve independently. In this paper we give sufficient conditions for the convergence of the first moments of the slow variable in the coupled case. Our proof is based upon a new method of stochastic regularization and functional-analytical techniques combined via a double limit procedure involving a zero-noise limit as well as considering $\epsilon$ to zero. We also give exact formulas for the drift and diffusion coefficients for the limiting SDE. As a main application of our theory, we study \emph{weakly-coupled} systems, where the coupling only occurs in lower time scales and our conditions are more easily verifiable requiring only mild, namely summable, decay of correlations.
\end{abstract}



\section{Introduction}
\label{sec:intro}

Many natural processes can be modeled by systems with two clearly separated sets of variables: 
a set of variables which evolve rapidly in time (for instance, within milliseconds) and a set of 
slowly varying variables (for instance, variables for which change is observed after hundreds of 
years); see~\cite{44} for many examples and techniques in fast-slow systems. In many applications 
the rapidly varying variables lie in a high-dimensional space and complicate the model significantly.
Typical examples are chemical processes such as combustion~\cite{MaasPope1}, or climate dynamics~\cite{70}. Therefore, one naturally seeks reduced equations 
for the slow dynamics only. Several formal and rigorous reduction methods exist, such as 
Fenichel-Tikhonov slow manifolds~\cite{Fenichel4,44, Tikhonov}, averaging~\cite{Verhulst} and 
homogenization~\cite{51,PavliotisStuart}. 

In this paper we are going to study multiscale ordinary 
differential equations (ODEs) with three separated time scales and fast chaotic dynamics: firstly, 
a fast time scale $\mathcal{O}(\epsilon^2)$ with nontrivial fast chaotic dynamics, but with slow dynamics 
which are practically in equilibrium, secondly an intermediate time scale $\mathcal{O}(\epsilon)$ with fast 
dynamics which have equilibrated, and finally a slow time scale $\mathcal{O}(1)$ (diffusive time scale). 
When the slow variables start to evolve under the influence of the fast dynamics, one observes induced 
fluctuations. 
In this setting, the method of reduction to a single slow equation is usually called 
\emph{homogenization}. Common techniques to achieve the reduction include methods based upon partial differential 
equations (PDEs) via the Liouville or Fokker-Planck/Kolmogorov equations~\cite{37,43}, techniques based upon 
semigroups~\cite{47}, algorithmic approaches~\cite{35}, as well as pathwise approaches via dynamical systems 
and probabilistic limit laws which we will focus on: in recent years, Melbourne and co-workers~\cite{14,13,12,20} have obtained rigorous convergence 
results, with high generality and mild assumptions, for the slow process $x_\epsilon$ within fast-slow systems of the form 
 \begin{subequations}
\label{eq: fast slow ODE only y}
 	\begin{alignat}{4}
 	\dot{x}_\epsilon &= a(x_\epsilon,y_\epsilon) +\epsilon^{-1}b(x_\epsilon,y_\epsilon), \text{\quad }
 	x_\epsilon(0;\eta) = \xi \in \mathbb{R}^d , \text{ for all $\eta \in \Omega$}, &&\text{ (slow equation),} 
	\label{eq: slow eq 1}\\%
 	\dot{y}_\epsilon &= \epsilon^{-2}g(y_\epsilon),\text{\quad}  y_\epsilon(0;\eta) = \eta \in \Omega\subset 
	\mathbb{R}^m,\text{ for all $\eta \in \Omega $, } &&\text{(fast equation)} \label{eq: fast eq 1},
 	\end{alignat}
 \end{subequations}
where the vector fields $a:\mathbb{R}^d\times\mathbb{R}^m\to\mathbb{R}^d$, $b:\mathbb{R}^d \times \mathbb{R}^m 
\to \mathbb{R}^{d}$ are $C^3$ and bounded with globally bounded derivatives. A main dynamical assumption
is to require ergodicity for the fastest scale, i.e., the ODE $\dot{y} = g(y)$, $y\in \mathbb{R}^m$, generates a 
flow $\phi_t: \mathbb{R}^m \to \mathbb{R}^m $ with a compact invariant set $\Omega\subset \mathbb{R}^m$ and ergodic 
invariant probability measure $\mu$ supported on $\Omega$. 
Another intrinsic part of this setup is the centering condition
 \begin{equation*}
 \int_\Omega b(x,y) ~\rmd\mu(y) = 0, \text{\quad for all }  x\in \mathbb{R}^d.
 \end{equation*}
 Systems of the form~\eqref{eq: fast slow ODE only y} are also called skew products, because they are not 
coupled but instead the fast variables $y_\epsilon$ can be described by a separate dynamical system on $\Omega$. 
Further, we note that the initial condition $\eta$ is the only source of randomness in the system. Without particular mixing conditions on the flow $\phi_t$, Kelly and Melbourne have shown~\cite{12} that for any finite $T>0$ the slow process $x_\epsilon$ converges 
weakly in $C([0,T], \mathbb{R}^d)$ to the solution $X$ of an It\^o stochastic differential equation (SDE) of the form
\begin{equation}
     \rmd X = \tilde{a}(X) ~\rmd t + \sigma(X)~ \rmd W, \text{\quad} X(0) = \xi,
\end{equation}
where $W$ is an $\mathbb{R}^d$-valued standard Brownian motion, $\sigma$ is a matrix-valued map and $\tilde a$ denotes a modified drift term. Mixing assumptions on the flow $\phi_t$ are needed for more specific formulas for drift and diffusion coefficients. 

Although one might intuitively expect that fast chaotic noise may be approximated by a stochastic process, it is neither
obvious which stochastic integral to consider nor how to prove the convergence to an SDE. The main difficulty lies in
the fact that fast-slow systems are singular perturbation problems~\cite{44} as $\epsilon\ra 0$. Yet, as described above, there even exist 
exact formulas for the drift term $\tilde{a}: \mathbb{R}^d \to \mathbb{R}^d$ and the diffusion coefficient $\sigma: 
\mathbb{R}^d \to \mathbb{R}^{d\times d}$. However, the skew-product structure~\eqref{eq: fast slow ODE only y} is a big
practical restriction as it is well-known that in most applications, the fast and slow variables are 
coupled~\cite{44}. Our main goal in this paper is to study coupled deterministic fast-slow systems or, in other 
words, to generalize the study of systems of the form \eqref{eq: fast slow ODE only y} by considering the case $g = g(x,y)$. Unlike skew products, coupled systems have barely been covered in the literature, with the only results for the discrete-time case being obtained by Dolgopyat in \cite{dolgopyat}, according to our best knowledge.  
Informally speaking, we are going to prove that as $\epsilon\ra 0$, the solutions of the fast-slow ODE are 
well-approximated by an effective slow SDE; see Section~\ref{ssec:mainres} for precise statements. Our strategy to achieve 
this result is to employ a double singular limit argument via an intermediate small-noise regularization, i.e., the idea is
to pass to the stochastic level as early as possible in the proof and then use functional-analytic a-priori bounds to carry 
out both of the necessary limits. The specific proofs will need limits of the respective integrals for the coefficients such that mixing assumptions have to be made; this is the price we pay to show such results for the coupled case.
 	
\subsection{Main setup and strategy for coupled systems}
\label{sec main setup} 

More precisely, in this paper we are interested in coupled fast-slow systems of the form 
\begin{subequations}
\label{eq: fast slow ODE a and y}
	\begin{alignat}{4}
	\dot{x}_\epsilon &= a(x_\epsilon,y_\epsilon) +\epsilon^{-1}b(x_\epsilon,y_\epsilon), 
	\text{\quad } x_\epsilon(0;\eta) = \xi \in \mathbb{R}^d , \text{ for all $\eta \in \mathbb{T}^m$,} 
	&&\text{ (slow equation),} \label{eq: coupled sy, slow eq 1}\\
	\dot{y}_\epsilon &= \epsilon^{-2}g(x_\epsilon,y_\epsilon), \text{\quad} y_\epsilon(0;\eta) 
	= \eta \in \Omega\subset \mathbb{T}^m,\text{ for all $\eta \in \mathbb{T}^m$, } &&\text{(fast equation)}.
	\label{eq: coupled sy, fast eq 1}
	\end{alignat}	
\end{subequations}
Before we can provide our main results, we state several assumptions, which are supposed to hold:
\begin{assms} 
\label{assumpt. cont time coupled systems}
	\begin{enumerate}
	\item[(A1)] The functions $a:\mathbb{R}^d\times\mathbb{T}^m\to\mathbb{R}^d$ , $b:\mathbb{R}^d 
	\times \mathbb{T}^m \to \mathbb{R}^{d}$ are $C^3$ with globally bounded derivatives up to order one. 
	\item[(A2)] For every fixed $x \in \mathbb{R}^d$, when viewed as a parameter, the ODE
	$\dot{y} = g(x,y)$ ,  $y\in \mathbb{T}^m$,  generates a flow $\phi_x^{0,t}: \mathbb{T}^m \to 
	\mathbb{T}^m $ with a compact invariant set $\Omega\subset \mathbb{T}^m$ and ergodic invariant 
	probability measure $\mu_x^0$ supported on $\Omega$. Furthermore, $g$ is $C^3$ with globally bounded 
	derivatives up to order two. 
	\item[(A3)] For the function $b(x,\cdot): \Omega \to \mathbb{R}^{d}$, the following centering condition 
	is satisfied:
	\begin{equation}
	\label{eq: centering cond for d=0}
	\int_{\Omega} b(x,y) ~\rmd\mu_x^0(y) = 0 \text{\quad for all }  x\in \mathbb{R}^d.
	\end{equation}
	\end{enumerate}
\end{assms} 
Due to the coupling, the argument used for skew products cannot be repeated (cf.~Section \ref{sub: main idea 
used in pr res}) and we need a new ansatz. Our strategy is the following: 
\begin{enumerate}
	\item Instead of proving weak convergence of the slow process (as a measure in $C([0,1],\mathbb{R}^d)$), we first 
	try to prove a weaker form of convergence (e.g.~convergence in distribution at any time).
	\item We add small stochastic non-degenerate noise to the fast subsystem in order to use results on uniformly 
	elliptic SDEs.
	\item We let the noise in the stochastic system tend to zero and find the right limiting behaviour for the 
	deterministic fast-slow system.
\end{enumerate}
The main reason, why we choose to work with stochastic systems as an intermediate step is that they provide a 
regularization. The infinitesimal generator for the semigroup of the associated Kolmogorov 
equation is uniformly elliptic. In particular, this case has been studied and weak convergence of the slow 
process has been rigorously proven. Such systems have the form 
\begin{subequations}
\label{eq: fast slow SDE}
	\begin{alignat}{4}
	\frac{\rmd x_{\epsilon,\delta}}{\rmd t} &= a(x_{\epsilon,\delta},y_{\epsilon,\delta}) + \frac{1}{\epsilon}
	b(x_{\epsilon,\delta},y_{\epsilon,\delta}), \text{\quad } x_{\epsilon,\delta}(0) = x_0, &&\text{ (slow equation),} 
	\label{eq: slow eq sde 1}\\
	\frac{\rmd y_{\epsilon,\delta}}{\rmd t} &= \frac{1}{\epsilon^2}g(x_{\epsilon,\delta},y_{\epsilon,\delta}) + 
 \frac{1}{\epsilon} \sqrt{\delta}\frac{\rmd V}{\rmd t}, \text{\quad} y_{\epsilon,\delta}(0) = y_0,  &&\text{(fast equation)}. 
\label{eq: fast eq sde 1}
	\end{alignat}	
\end{subequations}
Here it is always assumed that $\delta>0$, $V$ is an $m$-dimensional Brownian motion on a probability space 
$(\Lambda,\mathcal{F},\nu)$ and the SDE is to be understood as an integral equation, as usual, where $\frac{\rmd V}{\rmd t}$ 
denotes white noise viewed as the usual generalized stochastic process~\cite{ArnoldSDEold}. Further, let $\mathbb{E}$ denote 
the expectation with respect to the Wiener measure $\nu$. It is well-known that for a sufficiently smooth function 
$v: \mathbb{R}^d\times \mathbb{T}^m\to \mathbb{R}$ the first moments 
\begin{equation*}
 u^{\epsilon,\delta}(x,y,t):= \mathbb{E}[v(x_{\epsilon,\delta}(t), y_{\epsilon,\delta}(t))|(x_{\epsilon,\delta}(0), 
y_{\epsilon,\delta}(0)) = (x,y)]
\end{equation*}
satisfy the backward Kolmogorov equation
\begin{equation}
\label{eq: bachward kolm eq with eps}
 \frac{\rmd u^{\epsilon,\delta}}{\rmd t} = \mathcal{L}^{\epsilon,\delta}u^{\epsilon,\delta} 
:= \Big(\frac{1}{\epsilon^2} \mathcal{L}_1^\delta + \frac{1}{\epsilon}\mathcal{L}_2 + \mathcal{L}_3 \Big)
u^{\epsilon,\delta},
 \end{equation}
 where
 \begin{gather*}
 \mathcal{L}_1^\delta u := g \cdot\nabla_y u + \frac{1}{2} \delta I: \nabla_y \nabla_y u, \\
 \mathcal{L}_2 u :=    b\cdot\nabla_x u, \\
 \mathcal{L}_3 u :=   a  \cdot\nabla_x u.
 \end{gather*}
Here we use the notation $A:B = \textnormal{trace}(A^\top B)= \sum_{ij}a_{ij}b_{ij}$ for the inner product 
of two  matrices $A$ and $B$, $\nabla$ for the gradient and $\nabla\nabla$ for the Hessian matrix. 
Note that (see for example \cite[Chapter 11]{PavliotisStuart}) the operator $\mathcal{L}_1^\delta: D(\mathcal{L}_1^\delta) 
\subset L^2(\mathbb{T}^m) \to L^2(\mathbb{T}^m)$ is uniformly elliptic and has for every fixed $x \in 
\mathbb{R}^d$, viewed as a parameter, a one-dimensional null space. The null space is characterized by
\begin{equation}
\label{eq: erg ass 1}
\begin{split}
\mathcal{L}_1^\delta C &=0, \\
\Big(\mathcal{L}_1^\delta\Big)^* \rho^\delta_\infty(y;x) &= 0, 
\end{split}
\end{equation}
where $C$ denotes the constant functions in $y$ and $\rho^\delta_\infty$ is the Lebesgue density of the measure 
$\mu_x^\delta$, i.e., 
\begin{equation}
\label{eq: measure mu_x for stochastic sy}
\rmd\mu_x^\delta(y) := \rho^\delta_\infty(y;x) ~\rmd\lambda^m(y),
\end{equation}
where $\mu_x^\delta$ is the unique ergodic invariant measure of the SDE
\begin{equation*}
\frac{\rmd y}{\rmd t} = g(x,y) + \sqrt{\delta} \frac{\rmd V}{\rmd t}.
\end{equation*} 
Assume additionally that the centering condition 
\begin{equation} 
\label{eq: centering condition in theorem}
\int_{\mathbb{T}^m} b(x,y) \rho_\infty^\delta(y;x) ~\rmd y = 0
\end{equation}
is satisfied for all $x \in \mathbb{R}^d$ and $\delta > 0$.
Then, due to the uniform ellipticity of $\mathcal{L}_1^\delta$ for
$\delta > 0$, applying the Fredholm alternative \cite[Theorem 7.9]{PavliotisStuart} gives the existence of a unique 
centered solution $\Phi^\delta(y;x)$ of the so-called cell problem
\begin{equation}
\label{eq: cell problem}
- \mathcal{L}_1^\delta \Phi^\delta(y;x) = b(x,y), \text{\quad}  \int_{\mathbb{T}^m} \Phi^\delta(y;x) 
\rho_\infty^\delta(y;x) ~\rmd y = 0.
\end{equation}
Using perturbation expansion techniques, which we will discuss in more details in 
Section \ref{sec: basic results for stochastic systems}, it can been shown that $u^{\epsilon,\delta}$ can 
be approximated by the leading order component $u_0^\delta$ which satisfies 
\begin{equation}
\label{eq L^0,delta}
\frac{\rmd u_0^\delta}{\rmd t} = \mathcal{L}^{0,\delta}u_0^\delta,
\end{equation}
where the operator $\mathcal{L}^{0,\delta}$ acts on the twice continuously differentiable functions with 
compact support $C^2_\txtc(\mathbb{R}^d)$ via
\begin{equation} 
\label{eq: generator L0d}
 \mathcal{L}^{0,\delta}u := F^\delta(x) \cdot \nabla_{x} u + \frac{1}{2}  
A^\delta(x)A^\delta(x)^\top : \nabla_x \nabla_{x} u,
\end{equation}
where the coefficients $F^\delta$ and $A^\delta$ depend on the solution $\Phi^\delta$ of the cell 
problem~\eqref{eq: cell problem} and are given by
\begin{equation}
\label{eq: F and A 1}
\begin{split}
F^\delta(x) &:= \int_{\mathbb{T}^m} \Big(a(x,y) + (\nabla_x \Phi^\delta(y;x))b(x,y) \Big)\rho_\infty^\delta(y;x) 
~\rmd y \\
&= F_1^\delta(x) + F_0^\delta(x), \\
A^\delta(x)A^\delta(x)^\top  &:=\frac{1}{2}\Big(A_0^\delta(x) + A_0^\delta(x)^\top  \Big),\\
A_0^\delta(x) &:= 2\int_{\mathbb{T}^m} b(x,y) \otimes \Phi^\delta(y;x) \rho_\infty^\delta(y;x)~\txtd y. 
\end{split}
\end{equation}
We are now ready to state our main theorems. 

\subsection{Main results}
\label{ssec:mainres}

In the following, let $(X^\epsilon(t;\xi ,\eta), Y^\epsilon(t;\xi,\eta))$ denote the solution of the 
ODE~\eqref{eq: fast slow ODE a and y} for any $\epsilon > 0$ and let $C_0(\mathbb{R}^d)$ denote the 
space of continuous functions vanishing at infinity, i.e., as $\|x\|\ra \I$. Note that we still use the 
notation of Section \ref{sec main setup}. In addition we assume:
\begin{enumerate}
\item[(A4)] There exists a generator $\mathcal{L}^{0, 0}$ of a strongly continuous semigroup $T^{0,0}$ on $C_0(\mathbb{R}^d)$, with domain $D\subset C_0(\mathbb{R}^d)$ containing  $C^2_\txtc(\mathbb{R}^d)$, such that for all $f \in C^2_\txtc(\mathbb{R}^d)$ we have
	\begin{equation}
	\label{eq: led to loo 7} 
	\lim_{\delta \to 0}\mathcal{L}^{0, \delta} f = \mathcal{L}^{0, 0}f  \text{\quad uniformly.}
	\end{equation}
\end{enumerate}

\begin{theorem} 
\label{thm A}
Assume (A1)-(A4). Then, for 
every $f \in C_0(\mathbb{R}^d)$ and every sequence $\{\epsilon_k\}_{k\geq 0}$ 
with $\epsilon_k \to 0$ for $k \to \infty$, there exists a subsequence 
$\{\epsilon_{k_m}\}_{m \geq 0}$ such that for $m \to \infty$
\begin{equation*}
f(X^{\epsilon_{k_m}}(t; \xi , \eta)) \to T^{0,0}(t) f (\xi), 
\text{\quad uniformly in $\xi \in \mathbb{R}^d$, $\eta \in \Omega$ 
	and $t\in [0,\hat{T}]$},
\end{equation*}
where 
$\hat{T}$ is any finite time.
\end{theorem}

Theorem~\ref{thm A} provides a convergence result of the original fast-slow
system with sufficiently strong assumptions on the fast chaotic dynamics to a Markov process, whose correspondence with a reduced
slow SDE is specified below in the context of Theorem B (see \eqref{Limiting SDE in corollary after Led Loo}). The notion of convergence is to be understood in a weak averaged sense but it
does cover the coupled case. 
The proof of Theorem \ref{thm A} is provided in Section \ref{sec: main result}.  
The second main result, Theorem \ref{thm B}, gives sufficient conditions under which the main 
assumption (A4) 
 in Theorem \ref{thm A} is satisfied.
 Let us define the 
solution operator $\phi_x^{\delta,t}(y)$ of the fast equation for $\epsilon = 1$, solving, for 
a fixed $x \in \mathbb{R}^d$, the SDE
\begin{equation} 
\label{eq: phi_xi}
\frac{\rmd}{\rmd t} \phi_x^{\delta, t}(y) = g(x,\phi_x^{\delta, t}(y)) + \sqrt{\delta} 
\frac{\rmd V}{\rmd t}, \text{\quad} \phi_x^{\delta,0}(y)= y.
\end{equation}
Note that $\phi_x^{\delta,t}(y)$ depends on a Brownian motion and, hence, is a stochastic 
process $\phi_x^{\delta,t}(y)(\omega)$, $\omega \in \Lambda$. Furthermore, notice that the 
flow $\phi_x^{0,t}$ is purely deterministic.
%
\begin{theorem}
\label{thm B}
Assume that the unperturbed flow $\phi_x^{0,t}$ has an ergodic invariant probability 
measure $\mu^0$ and summable stochastically stable decay of correlations $C(t;x)$ in the 
sense of Definitions \ref{defin: doc} and \ref{defin: stochastic stability}.
Additionally (A1)-(A2) are satisfied and suppose the following centering condition holds
\begin{equation} 
	\label{eq: centering condition with delta}
	\int_{\mathbb{T}^m} b(x,y) ~\rmd\mu_x^\delta(y) =0 \text{\quad for all } 
	x \in \mathbb{R}^d \text{ and } \delta \geq 0.
\end{equation}
Then we have the following:
\begin{enumerate}
\item In the case that $g=g(y)$ is independent of $x$,	then condition (A4) is satisfied.
	\item In the general case that $g=g(x,y)$, (A4) 
	holds provided that the centering condition 
	\begin{equation} 
	\label{eq: lem convergence of F_0 in the case gxy 2}
	\int_{\mathbb{T}^m}  \nabla_y b(x,y) ~\rmd\mu_x^\delta(y) = 0 
	\text{\quad for all } x \in \mathbb{R}^d \text{ and } \delta \geq 0.,
	\end{equation}
	and the growth assumption 
	\begin{equation} 
	\label{eq: lem convergence of F_0 in the case gxy 1}
	\int_0^\infty  \sup_{x \in \mathbb{R}^d} \Big\{C(t; x)  
	\parallel\nabla_x \phi_x^{0,t}(\cdot) b(x,\cdot) \parallel_\alpha \Big\}~\rmd t < \infty 
	\end{equation}	
	are satisfied (Here, $\parallel \cdot \parallel_{\alpha}$ denotes the $\alpha$-H\"{o}lder norm for an $\alpha>0$). 
	\item The operator $\mathcal{L}^{0,0}$ can be written 
	as \begin{equation}
	\label{eq: Loo in explicit form}
		\mathcal{L}^{0,0}u = F^0(x) \cdot \nabla_x + \frac{1}{2}A^0(x)A^0(x) : \nabla_x \nabla_x u,
	\end{equation}
	where the diffusion coefficient $A^0$ is given by	
	\begin{equation} 
	\label{eq: thm coupled systems A0}
	\begin{split}
	A^0(x)A^0(x)^\top  &=  \frac{1}{2} \Big( A_0^0(x) + A_0^0(x)^\top  \Big), \\
	A_0^0(x) &= 2 \int_0^{\infty}  \lim_{T \to \infty} \frac{1}{T} \int_0^T  
	b(x, \phi_x^{0,s}(y)) b \Big(x,\phi_x^{0,t+s}(y)\Big)~ \rmd s  ~\rmd t.
	\end{split}
	\end{equation}
and the drift term $F_0$ is given by 
	\begin{align}
	\label{eq: thm coupled systems F0}
	F^0 (x) = \lim_{T \to \infty} \frac{1}{T} \int_0^T a(x,\phi_x^{0,s}(y)) \rmd s \nonumber \\+ 
	\lim_{T \to \infty} \frac{1}{T} \int_0^T  \Bigg(\nabla_x b\Big(x &,\phi_x^{0,t+s}(y) \\ &
	+ \nabla_yb\Big(x,\phi_x^{0,t+s}(y)\Big)\nabla_x\phi_x^{0,t}(\phi_x^{0,s}(y)) \Bigg) 
	b\Big(x,\phi_x^{0,s}(y) \Big) ~\rmd s. \nonumber
	\end{align}
\end{enumerate}
\end{theorem}	

Theorem \ref{thm B} is proven at the end of Section \ref{sec: convergence of Lod}. Note that the Markov process $X$ generated by $\mathcal{L}^{0,0}$ is expliticitly given by the SDE
	\begin{equation}
	\label{Limiting SDE in corollary after Led Loo}
	\rmd X = F^0(X) + A^0(X)~\rmd W, \text{\quad } X(0) = \xi \in \mathbb{R}^d,
	\end{equation}
	whose unique solvability is guaranteed by the smoothness and boundedness assumptions (A1), (A2). Moreover, the action of the semigroup $T^{0,0}f$ is given by 
	$\mathbb{E}[f(X(t))]$.  
The growth 
assumption  \eqref{eq: lem convergence of F_0 in the case gxy 1} is a strong mixing assumption 
on the flow and it remains to be determined precisely how large the class of functions satisfying 
this property is in applications (see remarks in Section \ref{sec: main result}). One possible way to 
weaken this assumption is to consider systems that are not coupled in the strongest possible 
sense, but for which the coupling occurs in smaller time scales. We refer to such systems as 
\textit{weakly-coupled} and their general form is given by the following fast-slow ODE 
on $\mathbb{R}^d \times \mathbb{T}^m$ 
\begin{subequations}
\label{eq: coupl in slower time scales 0DE 1 }
\begin{alignat}{4}
\frac{\rmd x_\epsilon}{\rmd t} &= a(x_\epsilon, y_\epsilon) + \frac{1}{\epsilon} 
b(x_\epsilon, y_\epsilon),  \text{\quad} x_\epsilon(0) = \xi, \\
\frac{\rmd y_\epsilon}{\rmd t} &= \frac{1}{\epsilon^2} g(y_\epsilon) + \frac{1}{\epsilon} 
h(x_\epsilon,y_\epsilon) + r(x_\epsilon,y_\epsilon), \text{\quad}y_\epsilon(0) = \eta.
\end{alignat}
\end{subequations}
Indeed, there are several examples of multiscale systems with interesting dynamical behaviour such as mixed-mode oscillations, where three time scales occur (see for example \cite{Popovic2016, KrupaPopovicSIAM, KrupaPopovicchaos}). Furthermore, these three-scale systems are often very similar to related problems of van der Pol type, where rigorous proofs for chaos exist~\cite{Haiduc1}. 

In the following, let $(X^\epsilon(t;\xi ,\eta), Y^\epsilon(t;\xi,\eta))$ be the solution 
of the ODE~\eqref{eq: coupl in slower time scales 0DE 1 }. In this case, the solution operator 
$\phi^{\delta,t}$ for the fast dynamics of the stochastically perturbed system, given by
\begin{equation}\label{eq: solution oper phi without x}
\frac{d}{\rmd t} \phi^{\delta,t}(y) = g(\phi^{\delta,t}(y)) + \sqrt{\delta} ~\rmd V, 
\text{\quad } \phi^{\delta,0}(y) = y,
\end{equation}
does not depend on $x$. 

\begin{theorem} 
\label{thm C}
Assume (A1)-(A2) and 
\begin{enumerate}
	\item that the unperturbed flow $\phi^{0,t}$ has an ergodic invariant probability 
	measure $\mu^0$, summable and stochastically stable decay of correlations $C(t)$ in the sense 
	of Definitions \ref{defin: doc} and \ref{defin: stochastic stability},
	and that the centering condition~\eqref{eq: centering condition with delta} is satisfied,
	\item in the case that $h$ does not vanish everywhere, additionally, that the centering 
	condition~(\ref{eq: lem convergence of F_0 in the case gxy 2}) and 
	the growth condition 
	\begin{equation} 
	\label{eq: growth cond in coupl sys thm}
	\int_0^\infty C(t) \sup_{x \in \mathbb{R}^d} \Big\{ \parallel\nabla_y \phi^{0,t}(\cdot) h(x,\cdot)\parallel_\alpha \Big\} ~\rmd t <\infty
	\end{equation}
	are both satisfied.
\end{enumerate}	
Then, \begin{enumerate}
	\item condition (A4) is satisfied and for every $f \in C_0(\mathbb{R}^d)$ and every sequence $\{\epsilon_k\}_{k\geq 0}$ 
	with $\epsilon_k \to 0$ for $k \to \infty$, there exists a subsequence $ \{ \epsilon_{k_m} \}_{m \geq 0}$ 
	such that 
	\begin{equation*}
	f(X^{\epsilon_{k_m}}(t; \xi , \eta)) \to T^{0,0}(t) f(\xi), 
	\text{\quad uniformly in $\xi \in \mathbb{R}^d$, $\eta \in \Omega$ 
		and $t\in [0,\hat{T}]$}.
	\end{equation*}
	\item  The operator $\mathcal{L}^{0,0}$ can be written 
	as \begin{equation}
	\label{eq: Loo in explicit form 2}
	\mathcal{L}^{0,0}u = \tilde{F}^0(x) \cdot \nabla_x + \frac{1}{2}A^0(x)A^0(x) : \nabla_x \nabla_x u,
	\end{equation}
	where $\tilde{F}^0$ is given by (\ref{eq: thm weakly coupled systems F0}) and $A^0$ is given by  (\ref{eq: thm weakly coupled systems A0}).
\end{enumerate}

\end{theorem}

The proof of Theorem \ref{thm C} is given with Theorem \ref{thm: coupled systems} below. Note once again that $T^{0,0}(t)f =\mathbb{E}[f(X(t))]$, where the Markov process $X$ is generated by $\mathcal{L}^{0,0}$. Moreover, $X$ solves the SDE (\ref{Limiting SDE in corollary after Led Loo}) (with modified drift $\tilde{F}^0$ instead of $F^0$).  
Basically Theorem~\ref{thm C} states that we have the desired convergence, where the growth
assumption on the correlation function is relaxed in the sense that weakly-coupled fast-slow
systems behave more like the skew-product case. More precisely, for weakly-coupled systems of the form (\ref{eq: coupl in slower time scales 0DE 1 }),
\begin{itemize}
	\item  with vanishing $h \equiv 0$ (i.e. with coupling occuring only in the lowest posssible time scale), summable decay of correlations (DOC) is sufficient, provided the natural uniformity property of Definition~\ref{defin: stochastic stability} is satisfied (see Remark~\ref{rem: stoch_stable_doc}). Examples for such systems include Anosov flows with exponential DOC, like for instance geodesic flows on compact
	negatively curved surfaces \cite{DOC1} or contact Anosov flows \cite{DOC2}, Axiom A flows with superpolynomial DOC (also called rapid mixing)  \cite{DOC3} 
	or non-hyperbolic flows with a stable $C^{1+\alpha}$ foliation including some geometric Lorenz attractors \cite{DOC6}.  See also Section \ref{sec: DOC}. 
	\item with non-vanishing $h$, the correlation function must satisfy the stronger assumption (\ref{eq: growth cond in coupl sys thm}).
\end{itemize}
 
In summary, our results provide an entire scale
of results from the more classical skew-product structure, via weak coupling to strong coupling.    

\begin{remark}
The explicit formulas for $A^0$ and $\tilde{F}^0$ for Theorem \ref{thm C} are
\begin{equation} 
\label{eq: thm weakly coupled systems A0}
\begin{split}
	A^0(x)A^0(x)^\top  &=  \frac{1}{2} \Big( A_0^0(x) + A_0^0(x)^\top  \Big), \\
	A_0^0(x) &= 2 \int_0^{\infty}  \lim_{T \to \infty} \frac{1}{T} \int_0^T  
	b(x, \phi^{0,s}(y)) b \Big(x,\phi^{0,t+s}(y)\Big) ~\rmd s  ~\rmd t.
	\end{split}
	\end{equation}
	and
	\begin{align}\label{eq: thm weakly coupled systems F0}
	\tilde{F}^0 (x) = \lim_{T \to \infty} \frac{1}{T} \int_0^T a(x,\phi^{0,s}(y)) 
	~\rmd s \nonumber \\+  \int_0^\infty \lim_{T \to \infty} \frac{1}{T} \int_0^T  
	\Bigg( & \nabla_x b\Big(x,\phi^{0,t+s}(y) \Big) b\Big(x,\phi^{0,s}(y)\Big)  \\ 
	&+ \nabla_y b\Big(x,\phi^{0,t+s}(y) \Big) \nabla_y \phi^{0,t} (\phi^{0,s}(y)) 
	h\Big(x,\phi^{0,s}(y) \Big) \Bigg) ~\rmd s ~\rmd t.\nonumber
	\end{align}
\end{remark}

\subsection{Outline of the paper}  
  
In Section \ref{sec: coupled systems} we first discuss the main idea of the proofs used 
in~\cite{13, 12} for proving weak convergence of the slow process in skew product systems 
(Section \ref{sub: main idea used in pr res}) (Section \ref{sub: main idea used in pr res}) and we also summarize some progress, which has been achieved over the last years, in proving mixing properties of certain classes of flows (Section \ref{sec: DOC}). We then recall and extend in 
Section \ref{sec: basic results for stochastic systems} some basic facts required for stochastic 
systems. In Section \ref{sec: main result}, we prove Theorem \ref{thm A}, which provides criteria 
to guarantee weak convergence of the slow process for coupled systems. In 
Section \ref{sec: convergence of Lod}, we then prove Theorem \ref{thm B}, which gives sufficient 
conditions for verifying the main assumption in Theorem \ref{thm A} and provides explicit 
formulas for the drift and diffusion coefficients of the limiting It\^{o} SDE. In 
Section \ref{sec: weakly coup sys} we apply our theory to weakly-coupled systems: we transfer the 
results obtained for coupled systems leading to the proof of Theorem \ref{thm C} 
(Section \ref{sub: main results for weakly coupled systems}) and, in addition, discuss a numerical 
example (Section \ref{sec: numerical example}). Finally, in Section \ref{sec: conclusion and outlook} 
we state our conclusions and discuss open problems and directions for further research.   

\section{From skew products to coupled systems}
\label{sec: coupled systems}

\subsection{Main idea used in previous results}
\label{sub: main idea used in pr res}

Before starting proving our main results, we want quickly summarize the main idea used in~\cite{13} 
and \cite{12} to study systems of the form (\ref{eq: fast slow ODE only y}). This provides suitable
background for the reader and also shows that our approach to the problem works along a completely
different route. The basic tool used in~\cite{13,12} is the so-called Weak Invariance Principle (WIP) 
and the idea of the proof can been very easily illustrated in the special case of a multiplicative 
noise (considered in \cite{13}), i.e., under the additional assumption that the vector-field $b$ has 
a multiplicative structure 
\begin{equation}
\label{eq: b has mult str}
b(x,y) \equiv b(x) v(y), \text{\quad } b: \mathbb{R}^d \to \mathbb{R}^{d \times e}, 
v: \Omega \to \mathbb{R}^e.
\end{equation}
For simplicity let us just in this section restrict to the case that the vector field $a$ is also 
independent of $y$, i.e., $a = a(x)$. In this case the system (\ref{eq: fast slow ODE only y}) can be 
rewritten as 
 \begin{equation} 
\label{eq: slow eq SDE}
 \rmd X_\epsilon = a(X_\epsilon)~\rmd t +  b(X_\epsilon)~\rmd W_\epsilon, \text{\quad} X_\epsilon(0;\eta)=\xi,
 \end{equation}
 where the family of random elements $W_\epsilon(\cdot;\eta) \in C([0,1],\mathbb{R}^e)$ is defined by
 \begin{equation} \label{eq: defin of We}
 W_\epsilon(t;\eta):= \epsilon v _{t\epsilon^{-2}}(\eta), \text{\quad} v_t(\eta) := \int_0^tv\circ \phi_s(\eta)\rmd s.
 \end{equation}
The key observation now is that if the flow  $\phi_s$ is sufficiently chaotic, then the process $W_\epsilon$ 
satisfies the WIP
\begin{equation}
\label{eq: WIP assumption}
W_\epsilon \to_w W \text{\quad in } C([0,1], \mathbb{R}^e) \text{ as } \epsilon \to 0,
\end{equation}
which is a generalization of the Central Limit Theorem. Therefore, we are already tempted to conclude weak convergence 
of the slow process $X_\epsilon$. The framework under which this intuitive idea has been rigorously justified is rough 
path theory~\cite{FrizHairer}. Equation~\eqref{eq: slow eq SDE} can be interpreted as a rough differential equation
 \begin{equation*}
 \rmd X_{\epsilon}  = a(X_\epsilon)~\rmd t + b(X_\epsilon)~\rmd(W_\epsilon,\mathbb{W}_\epsilon), 
\quad X_\epsilon(0;\eta) = \xi.
 \end{equation*}
 Noticing further, as shown in \cite{13}, that for any  $\gamma > \frac{1}{3}$ an iterated WIP, i.e.
 \begin{equation} \label{eq: iterated WIP}
 (W_\epsilon, \mathbb{W}_\epsilon) \to_w (W, \mathbb{W}) \text{\quad as $\epsilon \to 0$ in the rough path 
$\rho_\gamma$ topology},
 \end{equation}
 holds, one can conclude due to continuity of the solution map of such rough differential equations~\cite{FrizHairer} 
and the Continuous Mapping Theorem, the weak convergence of the slow process, i.e. as result of the form
 \begin{equation}\label{eq: intro Xe to X}
 X_\epsilon \to_w X \text{ as $\epsilon \to 0$, \quad} \rmd X = a(X)~\rmd t + b(X)*\rmd W, 
 \end{equation}
 where $b(X)*\rmd W$ is a certain kind of stochastic integral~\cite{13}. More general vector fields $b$ are considered 
in~\cite{12} and the main idea is to rewrite the system (\ref{eq: fast slow ODE only y}) in the form 
 \begin{equation*} 
 \rmd X_\epsilon = F(X_\epsilon) ~\rmd V_\epsilon + H(X_\epsilon)~\rmd W_\epsilon,
 \end{equation*}
 where $V_\epsilon$ and $W_\epsilon$ are function space valued paths given by
 \begin{equation*}
 V_\epsilon(t) = \int_0^t a(\cdot, y_\epsilon(r))~\rmd r \text{ and } 
W_\epsilon(t) = \epsilon^{-1} \int_0^t b(\cdot, y_\epsilon(r))~\rmd r.
 \end{equation*}
In this context, the operators $F(x), H(x)$ are interpreted as Dirac distributions located at $x$, 
that is $F(x) \phi = \phi(x)$ for any $\phi$ in the function space and similarly for $H$. 
Under mixing assumptions the iterated WIP (\ref{eq: iterated WIP}) holds and as in the case of 
multiplicative noise one can then conclude a result of the form (\ref{eq: intro Xe to X}). Exact 
formulas of the drift and diffusion coefficients are also given in \cite{12}. In summary, the 
approach relies upon a pathwise viewpoint and continuity in the rough-path topology to solutions
of ODEs/SDEs. Yet, this approach seems to be very difficult to generalize if the fast-slow 
system is fully coupled. In particular, this has motivated our approach to look for weaker 
convergence concepts in a more functional-analytic setting.  

\subsection{Rates of mixing for classes of flows}\label{sec: DOC}
In the following, we briefly give an overview over rigorous results on mixing rates of certain classes of flows that thereby satisfy summable decay of correlations in the sense of Definition~\ref{defin: doc}. Given a measure preserving flow $\phi_t: \Lambda \to \Lambda$, the \textit{correlation function} is defined as  
\begin{equation*}
\rho_{A,B}(t) := \int_{\Lambda} A \circ \Phi_tB d\mu - \int_{\Lambda} A d\mu \int_{\Lambda} B d\mu. 
\end{equation*} for \textit{observables} $A, B \in L^2(\Lambda, \mu)$.
The flow $\phi_t$ is called a \emph{mixing} if and only if $\rho_{A,B}(t)\to 0$ as $t \to \infty$ for all $A, B \in L^2(\Lambda, \mu)$ (see e.g.~\cite{DOC8}).

\subsubsection{Uniformly hyperbolic flows}
Assume that the flow $\phi_t: M \to M$ is  $C^2$ and defined on a compact manifold $M$. An invariant compact set $\Lambda \subset M$ is a \textit{hyperbolic set} for $\phi_t$, provided that the tangent bundle over $\Lambda$ admits a continuous $D\phi_t$- invariant spliting
\begin{equation*}
T_\Lambda M = E^u \oplus E^0 \oplus E^s
\end{equation*}
of uniformly contracting and expanding directions. 
For an \textit{Axiom A (uniformly hyperbolic) flow} the dynamics can be reduced into finitely many hyperbolic sets $\Lambda_1$, ... $\Lambda_k$, called \textit{hyperbolic basis sets}, which all contain a dense orbit. On every hyperbolic basic set $\Lambda= \Lambda_i$, for $i\in \{1,...,N \}$, we can associate, to every H\"{o}lder function on $\Lambda$ a unique invariant ergodic probability measure $\mu$.  We can further categorize Axiom A flows depending on the speed of mixing. For example, for flows with \textit{exponential DOC}, the correlation function, restricted to a suitable subspace of $L^2(\Lambda, \mu)$ (like, for example, an appropriate H\"{o}lder space), satisfies
\begin{equation*}
\rho_{A,B}(t) \leq C(A,B) e^{-\alpha t}, \quad \forall t>0,
\end{equation*}
for constants $C, \alpha>0$.  This was proven for example for certain classes of Anosov flows (i.e. special types of Axiom A flows for which the whole set $M$ is uniformly hyperbolic) like  geodesic flows on compact negatively curved surfaces \cite{DOC1} and contact Anosov flows \cite{DOC2}.  
Appart from exponential DOC we also have weaker notions, such as \textit{stretched exponential mixing}, i.e.~ for some constant $0 <\beta \leq 1$
\begin{equation*}
\rho_{A,B}(t) \leq C(A,B) e^{-\alpha t^\beta}, \quad \forall t>0,
\end{equation*}
which was proven for a large class of Anosov flows in dimension 3 \cite{DOC4}, 
 and \textit{superpolynomial decay} (or \textit{rapid mixing}), i.e.~for any $n>0$ the correlation function satisfies
 \begin{equation*}
 \rho_{A,B}(t) \leq C(A,B) t^{-n}, \quad \forall t>0,
 \end{equation*}
 or in other words, DOC at an arbitary polynomial rate. 
 Dolgopyat \cite{DOC5} proved rapid mixing for "typical" Axiom $A$ flows. Moreover, he has shown  that an open and dense set of Axiom A flows is rapid mixing, when restricted to sufficiently smooth observables \cite{dolgopyat}.  
 For all mentioned classes of mixing flows, the correlation is \textit{summable}, that is we have
 \begin{equation*}
\int_0^\infty \rho_{A,B}(t) dt < \infty. 
 \end{equation*}
\subsubsection{Non-uniformly hyperbolic flows}
Since the assumption of uniform hyperbolicity might be too restrictive for real applications, it is natural to seek for a good mixing theory for non-uniformly hyperbolic flows. Over the last few years remarkable progress has been achieved in this area; see e.g.~\cite{DOC8} and references therein for a good overview concerning results in this direction.  
For example, in \cite{DOC6}, extending results from \cite{DOC7}, exponential  DOC is proven for a class of non-uniformly hyperbolic skew-product flows satisfying an uniform integrability condition, which contains an open set of geometric Lorenz attractors.  Moreover, in \cite{DOC9}, for certain types of Gibbs-Markov flows, including intermittent solenoidal flows and various Lorentz gas
models including the infinite horizon Lorentz gas polynomial, DOC of the correlation function
\begin{equation*}
\rho_{A,B}(t) \leq C(A,B) t^{-(\beta -1)} \quad \forall t>0,
\end{equation*}
with $\beta>1$, is proven. For such flows, the DOC is summable, provided that $\beta>2$. 
\subsection{Basic facts for stochastic systems} 
\label{sec: basic results for stochastic systems}

Let us now come back to the coupled systems (\ref{eq: fast slow ODE a and y}). In the following we use 
the notation from Section \ref{sec main setup}. If we further consider the Banach space 
$X:= (C_0(\mathbb{R}^d\times \mathbb{T}^m) , \| \cdot \|_\infty)$ of continuous functions, 
which vanish as $\|x\|_2 \to \infty$ for points $(x,y) \in \mathbb{R}^d\times\mathbb{T}^m$; with the usual 
supremum norm, it can be shown (cf.~Lemma \ref{lem: suitable banach space L} in the Appendix) that 
the closure $\bar{\mathcal{L}_1}^\delta$ generates an ergodic strongly continuous contraction 
semigroup $\{ S^\delta(t) \}_{t\geq 0}$ on $X$ (in the sense of Definition~\ref{defin: ergodic contraction 
semigroup}) and $\bar{\mathcal{L}}^{\epsilon,\delta}$ generates 
a strongly continuous contraction semigroup on $X$ denoted by $\{T^{\epsilon,\delta}(t) \}_{t\geq 0}$. 
Let $\mathcal{P}^\delta$ be the projection corresponding to the ergodic semigroup produced by 
$\mathcal{L}_1^\delta$, acting on $X$ explicitly via
\begin{equation}
\label{eq: Projection Pd}
\mathcal{P}^\delta f (x,y) := \int_{\mathbb{T}^m} f(x,y) \rho_\infty^\delta(y;x) ~\rmd y, f\in X.
\end{equation}
The perturbation expansion 
\begin{equation} 
\label{eq. perturbation expansion for backward kolmogorov}
u^{\epsilon,\delta} = u_0^\delta + \epsilon u_1^\delta + \epsilon^2 u_2^\delta + \cdots,
\end{equation}
leads, as shown for instance in \cite{PavliotisStuart} and \cite{35} (cf.~Section \ref{sec: perturbation anal for wcsys app} 
in the Appendix for completeness) to the following equation for the leading order $u_0$:
\begin{align}
\label{eq: apr eq for barrho fp}
\frac{\rmd u_0^\delta}{\rmd t}(x,t) &=  \int_{\mathbb{T}^m}\rho_\infty^\delta(y;x)\mathcal{L}_3 
u_0^\delta(x,t) ~\rmd y - \int_{\mathbb{T}^m}\rho_\infty^\delta(y;x)\mathcal{L}_2\Big(\mathcal{L}_1^\delta\Big)^{-1}
\mathcal{L}_2 u_0^\delta(x,t)~\rmd y \nonumber\\
&= \Big(\mathcal{P}^\delta\mathcal{L}_3 \mathcal{P}^\delta -
\mathcal{P}^\delta\mathcal{L}_2\Big(\mathcal{L}_1^\delta\Big)^{-1}\mathcal{L}_2 \mathcal{P}^\delta\Big)
u_0^\delta(x,t) \nonumber \\
&=: \mathcal{L}^{0,\delta} u_0^\delta.
\end{align}
The operator $\mathcal{L}^{0,\delta}$ acting on the right side of equation (\ref{eq: apr eq for barrho fp}) 
can be more precisely evaluated, using the function $\Phi^\delta$ defined in (\ref{eq: cell problem}). 
As shown in \cite{PavliotisStuart}, equation (\ref{eq: apr eq for barrho fp}) can be rewritten as
\begin{align} \label{eq: apr eq for barrho fp 2 }
\frac{du_0^\delta}{\rmd t} &= F^\delta(x) \cdot \nabla_{x} u_0^\delta + \frac{1}{2}  A^\delta(x)A^\delta(x)^\top  
: \nabla_x \nabla_{x} u_0^\delta \nonumber \\
&= \mathcal{L}^{0,\delta}u_0^\delta
\end{align}
where  the drift and diffusion coefficients are given by (\ref{eq: F and A 1}) and 
$\mathcal{L}^{0,\delta}u_0^\delta$ is given by  (\ref{eq L^0,delta}).

The major disadvantage of the formulas (\ref{eq: F and A 1}) is that they use the solution 
$\Phi^\delta$ of the cell problem  which is not well-posed for $\mathcal{L}_1^0$ or in other 
words, in the case that we work with purely deterministic systems. However, there are also some 
alternative expressions, which are more suitable for deterministic systems and are already 
proven in \cite{PavliotisStuart}, but which are for convenience included in the following 
Lemma~\ref{lem: alternative representations of the expressions}, since we require 
some minor changes. The alternative expressions use the solution operator $\phi_x^{\delta,t}(y)$ 
of the fast dynamics given by (\ref{eq: phi_xi}). Recall that $\mathbb{E}$ denotes the expectation 
with respect to Wiener measure $\nu$ on $\Lambda$ and further let $\mathbb{E}^{\mu_x \otimes \nu}$ 
denote the expectation with respect to the product measure 
$\mu_x^\delta \otimes \nu$,  where $\mu_x^\delta$ is the ergodic measure defined 
in (\ref{eq: measure mu_x for stochastic sy}).

\begin{lem} \textbf{(Differentiability of the solution operator with respect to $x$)} 
\label{lem: differentiability of the solution oper wrt to x} \\
There exists a version of the stochastic process $\phi_x^{\delta,t}$ such that for almost all (a.a.) 
$\omega \in \Lambda$ the function $x \to \phi_x^{\delta,t}$ is continuously differentiable for 
every $t$ and the differential $\nabla_{x}\phi_x^{\delta,t}(y) \in \mathbb{R}^{m\times d}$ satisfies 
the linear ODE
	\begin{equation}
	\frac{\rmd}{\rmd t} \nabla_{x}\phi_x^{\delta,t}(y) = \nabla_x g(x,\phi_x^{\delta,t}(y)) 
	+ \nabla_y g(x,\phi_x^{\delta,t}(y))\nabla_{x}\phi_x^{\delta,t}(y), \quad \nabla_x\phi_x^{\delta,0}(y) = 0.
	\end{equation}
\end{lem}
\begin{proof}
This follows from \cite[Theorem 4.2]{52}, where we set $v^x(t):= y +\sigma_2\frac{\rmd V}{\rmd t}$, 
$u:= x$ and $\rmd Z_s:= \rmd t$ such that $\phi_x(t) = v^x(t) + 
\int_0^tg(x,\phi_x(s))~\rmd Z_s$, and observe that all assumptions are satisfied since $g$ has 
bounded derivatives up to order two.
\end{proof}

\begin{lem} 
\textbf{(Alternative representations of the coefficients of the limiting SDE)} 
	\label{lem: alternative representations of the expressions} \\
Fix a $\delta > 0$. We have the following alternative formulas for the vector fields 
$F_0^\delta(x), F_1^\delta(x)$ and the diffusion matrix $A_0^\delta(x)$ from equation (\ref{eq: F and A 1}): 
For all $y\in \mathbb{T}^m$ and for a.a.~$\omega \in \Lambda$ we have
	\begin{equation}
	F_1^\delta(x) = \lim_{T\to \infty} \frac{1}{T} \int_0^{T} a \Big( x,\phi_x^{\delta,s}(y)(\omega) 
	\Big)~\rmd s\label{eq: altern repres 3}
	\end{equation}
	and 
	\begin{equation}\label{eq: altern repres 4}
	A_0^\delta(x) = 2 \int_0^{\infty}  \lim_{T \to \infty} \frac{1}{T} \int_0^T  
	b(x, \phi_x^{\delta,s}(y)(\omega))\otimes\mathbb{E}[ b \Big(x,\phi_x^{\delta,t}
	(\phi_x^{\delta,s}(y)(\omega)) \Big)] ~\rmd s  ~\rmd t, 
	\end{equation}
	and if there exists a constant $D(t)$ such that 
	\begin{equation} \label{eq: lemma alterntavie expressions welldefin cond}
	\nabla_x \Big( \mathbb{E}[b(x,\phi_x^{\delta,t}(y))]  \Big) \leq D(t), 
	\text{ for all $x \in \mathbb{R}^d,$ \quad} \int_0^\infty D(t) ~\rmd t <\infty,
	\end{equation}
	then, it holds also that
	\begin{align}\label{eq: altern repres 5}
	F_0^{\delta}(x)  =  \int_0^\infty \Bigg(\lim_{T \to \infty} \frac{1}{T} 
	\int_0^T \mathbb{E} \Big[ &\nabla_x b\Big(x,\phi_x^{\delta,t}(\phi_x^s(y)(\omega))\Big) \nonumber \\
	&+ \nabla_yb\Big(x,\phi_x^{\delta,t}(\phi_x^{\delta,s}(y)(\omega)) \Big)\nabla_x
	\phi_x^{\delta,t}(\phi_x^{\delta,s}(y)(\omega)) \Big]  b\Big(x,\phi_x^{\delta,s}(y)(\omega) \Big)~ \rmd s \Bigg) ~\rmd t.
	\end{align}	
\end{lem}

\begin{proof}
We follow the proof given in \cite[Chapter 11]{PavliotisStuart}. We first calculate
	\begin{align*}
	\Phi^\delta(y;x) &= \int_0^\infty (\txte^{\mathcal{L}_1^\delta t}b)(x,y) ~\rmd t && \text{(by  \cite[Result 11.8]{PavliotisStuart})
	} \\
	&= \int_0^\infty \mathbb{E} [ b(x,\phi_x^{\delta,t}(y)) ] ~\rmd t &&\text{(by \cite[Theorem 6.6]{PavliotisStuart})}
	\end{align*}
	Thus, using Fubini's theorem, 
	\begin{align*}
	A_0^\delta(x) &= 2\int_{\mathbb{T}^m} b(x,y) \otimes \Phi^\delta(y;x) \rho_\infty(y;x)~\rmd y\\
	&= 2 \int_{\mathbb{T}^m} b(x,y) \otimes \int_0^\infty \mathbb{E} [ b(x,\phi_x^{\delta,t}(y)) ] 
	~\rmd t ~ \rho_\infty^\delta(y;x)~\rmd y \\
	&= 2\int_0^\infty\int_{\mathbb{T}^m} b(x,y) \otimes \int_\Lambda  b(x,\phi_x^{\delta,t}(y)(\omega)) 
	~\rmd\nu(\omega) \rho^\delta_\infty(y;x) ~\rmd y ~\rmd t\\
	&= 2 \int_0^{\infty} \mathbb{E}^{\mu_x^\delta\otimes \nu}[b(x,y) \otimes b(x,\phi_x^{\delta,t}(y))] ~\rmd t.
	\end{align*}
Setting $h(x,y;t):= b(x,y) \otimes\mathbb{E} [b(x,\phi_x^{\delta,t}(y))]$ we get from Theorem~\cite[Theorem 6.16]{PavliotisStuart} 
that for a.a.~$\omega \in \Lambda$ we have
	\begin{equation*}
	\int_{\mathbb{T}^m}  h(x,y;t) \rho_\infty^\delta(y;x) ~\rmd y = \lim_{T \to \infty} \frac{1}{T} 
	\int_0^T h(x,\phi_x^{\delta,s}(y)(\omega))\rmd s 
	\end{equation*}
	and by inserting into the expression for $A_0^\delta(x)$ we get that for a.a. $\omega \in\Lambda$ 
	equation	(\ref{eq: altern repres 4}) is satisfied.	Analogously
	(noticing that condition (\ref{eq: lemma alterntavie expressions welldefin cond}) allows us to 
	interchange the 
	order of integration and the $\nabla_x$ operator),
	\begin{align*}
	F_0^\delta(x) &= \int_{\mathbb{T}^m} \rho_\infty^\delta(y;x)\nabla_x \Phi^\delta(y;x)b(x,y)~\rmd y\\
	&= \int_{\mathbb{T}^m} \rho_\infty^\delta(y;x) \nabla_x [\int_0^\infty \mathbb{E} [ b(x,\phi_x^{\delta,t}(y))] 
	~\rmd t] b(x,y) ~\rmd y\\
	&= \int_{\mathbb{T}^m} \rho_\infty^\delta(y;x)  [\int_0^\infty \int_{\Lambda}\nabla_x 
	\Big(b(x,\phi_x^{\delta,t}(y)(\omega)) \Big) ~\rmd\nu(\omega) ~\rmd t] b(x,y) ~\rmd y \\
	&= \int_0^\infty \int_{\mathbb{T}^m} \rho_\infty^\delta(y;x)\int_{\Lambda}\nabla_x 
	\Big(b(x,\phi_x^{\delta,t}(y)(\omega)) \Big) ~\rmd\nu(\omega)b(x,y) ~\rmd y ~\rmd t \\
	&= \int_0^\infty \mathbb{E}^{\mu_x^\delta\otimes \nu}[\nabla_x \Big( b(x,\phi_x^{\delta,t}(y))\Big) 
	b(x,y) ]~\rmd t.
	\end{align*}
	By the chain rule we have that
	\begin{equation*}
	\nabla_x \Big(b(x,\phi_x^{\delta,t}(y)) \Big) = 
	\nabla_xb(x,\phi_x^{\delta,t}(y)) + \nabla_yb(x,\phi_x^{\delta,t}(y))\nabla_x\phi_x^{\delta,t}(y).
	\end{equation*}
	Thus, setting 
	\begin{equation*}
	h(x,y;t) :=  \mathbb{E} [\nabla_xb(x,\phi_x^{\delta,t}(y)) + \nabla_yb(x,\phi_x^{\delta,t}(y))
	\nabla_x\phi_x^{\delta,t}(y) ] b(x,y),
	\end{equation*}
	we get equation~\eqref{eq: altern repres 5} by \cite[Theorem 6.16]{PavliotisStuart}. Now the expression for $F_1^\delta$ 
	follows directly from \cite[Theorem 6.16]{PavliotisStuart}.
\end{proof}

Finally, let $(T^{0,\delta}(t))_{t \geq 0}$ denote the corresponding semigroup of the generator 
$\mathcal{L}^{0,\delta}$ on $C_0(\mathbb{R}^d)$. The basic important fact that we use in the 
following is that the semigroup $(T^{\epsilon,\delta}(t))_{t \geq 0}$ converges towards $(T^{0,\delta}(t))_{t \geq 0}$ 
as $\epsilon \to 0$, as stated in Theorem \ref{thm: application of kurtz}, which has similarly been proven by 
Kurtz~\cite{47}, but is formulated and shown in the Appendix for the reader's convenience. We are now ready to 
state the main result of this section.

\subsection{Main result for coupled systems}
\label{sec: main result}

In the following, let $\{T^{\epsilon,0}(t)\}_{t \geq 0}$ denote the semigroup on $X$ generated by 
$\mathcal{L}^{\epsilon,0}$, which is defined as in (\ref{eq: bachward kolm eq with eps}) with $\delta = 0$. 
Similarly we consider the generator $\bar{\mathcal{L}}^{0,0}$ for the strongly continuous semigroup 
$T^{0,0}(t)$ on $C_0(\mathbb{R}^d)$.

\begin{thm}
\label{thm: from Led to L00}
Under the assumptions (A1)-(A4), it follows that for every $f \in C_0(\mathbb{R}^d)$ and every 
sequence $\{\epsilon_k\}_{k\geq 0}$ with $\epsilon_k \to 0$ for $k \to \infty$, there exists a 
subsequence $ \{ \epsilon_{k_m} \}_{m \geq 0}$ such that for any finite time $\hat{T}>0$
\begin{equation}\label{eq: led to loo 8} 
\lim_{m \to \infty} \sup_{0 \leq t \leq \hat{T}} \parallel T^{\epsilon_{k_m},0}(t)f  - 
T^{0,0}(t)f \parallel_\infty = 0.
\end{equation}
\end{thm}

\begin{proof}
Fix $f \in C_0(\mathbb{R}^d)$. We have by the triangle inequality
\begin{equation}\label{eq: ued - u00}
	\begin{split}
	\parallel T^{\epsilon, 0}(t) f - T^{0,0}(t) f  \parallel_\infty \leq \parallel 
	&T^{\epsilon, 0}(t)f - T^{\epsilon,\delta}(t)f \parallel_\infty + \parallel 
	T^{\epsilon,\delta}(t)f  - T^{0,\delta}(t)f \parallel_\infty \\  & \text{\quad}+ 
	\parallel T^{0,\delta}(t)f - T^{0,0}(t)f \parallel_\infty.
	\end{split}
	\end{equation}	
Further, due to the definition of the operator $\mathcal{L}_1^\delta$ we see immediately that 
for all $f \in \mathcal{D}(\mathcal{L}^{\epsilon,\delta})$  
	\begin{equation}\label{eq: led to loo 9} 
	\lim_{\delta \to 0}\mathcal{L}^{\epsilon, \delta} f = \mathcal{L}^{\epsilon,0} f \text{\quad uniformly}.
	\end{equation}
	Due to equations (\ref{eq: led to loo 9}) and  (\ref{eq: led to loo 7}) and by the Trotter-Kato Theorem 
	(see for example \cite[Theorem 4.8]{42}) we observe that for any fixed $\epsilon >0$ the first and the last 
	term on the right side of equation (\ref{eq: ued - u00}) can be made arbitrary small as $\delta \to 0$. 
	The second difference for any fixed $\delta > 0$ can be also made arbitrary small as $\epsilon \to 0$ due 
	to Theorem \ref{thm: application of kurtz}. To be more precise, let $ \{ \epsilon_k \}_{k \geq 0} $ be a 
	sequence with $\epsilon_k \to 0$ for $k \to \infty$. Then we can find for every $k \in \mathbb{N}$ a $\delta_k > 0$ 
	so that 
	\begin{equation*}
	\parallel T^{\epsilon_k, 0}(t)f - T^{\epsilon_k,\delta_k}(t)f \parallel_\infty + 
	\parallel T^{0,\delta_k}(t)f - T^{0,0}(t)f \parallel_\infty <  \frac{2\epsilon_k}{3}.
	\end{equation*}
	Moreover, for any $k \in \mathbb{N}$ we can fix an $l(k) \in \mathbb{N}$ so that 
	\begin{equation*}
	\parallel T^{\epsilon_{l(k)},\delta_k}(t)f - T^{0,\delta_k}(t)f \parallel_\infty < \frac{\epsilon_k}{3}.
	\end{equation*}
	In this way we get a subsequence $\{ \epsilon_{l(k)} \}_{k \geq 0 }$ for which 
	\begin{equation*}
	\parallel T^{\epsilon_{l(k)}, 0}(t) f - T^{0,0}(t) f  \parallel_\infty \leq \epsilon_k
	\end{equation*}
	holds. The claim now follows by taking the limit $k\to \infty$.
\end{proof} 

\begin{remark} \label{rem: after theorem from ledd to leoo}
 A sufficient condition for the key assumption (A4) to hold is that 
	\begin{equation}
	F_0^\delta \to F_0^0, \text{\quad}  F_1^\delta \to F_1^0, \text{\quad } A_0^\delta \to A_0^0 
	\text{\quad uniformly in }x,
	\end{equation}
provided that the expressions $F_0^0, F_1^0, A_0^0$ are well-defined, which requires sufficiently 
fast decay of correlations. Furthermore, Theorem \ref{thm B} gives us precise conditions under, 
which (A4) is satisfied. In the case that $g= g(y)$ is independent of $x$, the posed 
assumptions are relatively mild. 
\end{remark}

Next, recall that for $\epsilon > 0$ we denote by $(X^\epsilon(t;\xi ,\eta), Y^\epsilon(t;\xi,\eta))$ 
the solution of the ODE~(\ref{eq: fast slow ODE a and y}).	
	
\begin{cor} 
\label{cor: after thm Led to Loo}
Assume that (A1)-(A4) hold, that $\mathcal{L}^{0,0}$ can be written as in (\ref{eq: Loo in explicit form}) and that SDE (\ref{Limiting SDE in corollary after Led Loo}) has the solution $X(t)$. Then for every 
$f \in C_0(\mathbb{R}^d)$ and every sequence $\{\epsilon_k\}_{k\geq 0}$  with $\epsilon_k \to 0$ 
for $k \to \infty$ there exists a subsequence $ \{ \epsilon_{k_m} \}_{m \geq 0}$ such that for 
$m \to \infty$, 
	\begin{equation*}
	f(X^{\epsilon_{k_m}}(t; \xi , \eta)) \to \mathbb{E}[f(X(t;\xi))], 
	\text{\quad uniformly in $\xi \in \mathbb{R}^d$, $\eta \in \Omega \subset \mathbb{T}^m $ 
	and $t\in [0,\hat{T}]$},   
	\end{equation*} 
where the expectation $\mathbb{E}$ is taken with respect to the Wiener measure (defined on $\Lambda$) 
of the Brownian motion $W$. It follows especially that for any Borel probability measure $\mu$ on 
$\mathbb{T}^m$ we have
\begin{equation*}
	\mathbb{E}^{\mu} [f(X^{\epsilon_{k_m}}(t))]\to \mathbb{E}[f(X(t))]\text{\quad uniformly in $t \in [0,\hat{T}]$}
\end{equation*}
\end{cor}

\begin{proof}
The first statement follows immediately from Theorem \ref{thm: from Led to L00}, observing that 
$(T^{\epsilon,0}(t) f)(x) =  f(X^{\epsilon}(t; x)) $ and $(T^{0,0}(t) f) (x) = \mathbb{E}[f(X(t;x))]$. 
The last statement follows from the dominated convergence theorem.
\end{proof}

\begin{remark}
Note that if there exists a unique solution to the SDE (\ref{Limiting SDE in corollary after Led Loo}), then this is exactly the Markov process generated by $\mathcal{L}^{0,0}$, but Theorem \ref{thm A} does not necessarily need this restriction.
A sufficient condition for existence and uniqness of solutions of the SDE is global Lipschitz continuity of the drift and diffusion coefficients which follows in the more particular context of Theorems B and C via the ergodic formulas~\eqref{eq: thm coupled systems A0}, \eqref{eq: thm coupled systems F0}, \eqref{eq: thm weakly coupled systems A0}, \eqref{eq: thm weakly coupled systems F0} and Assumptions (A1), (A2). In general, we need Lipschitz continuity of the averaged vector field 
	\begin{equation*}
	\bar{a}(x):= \int_{\mathbb{T}^m} a(x,y) ~\rmd\mu_x^0(y),
	\end{equation*}
which demands sufficiently smooth dependence of the invariant measures $\mu_x$ on the parameter $x$. 
This can be violated, if for example the fast dynamics exhibits bifurcations upon varying $x$. In fact, 
even continuity of $\bar{a}$ cannot be guaranteed in such cases. The problem of non-smooth dependence 
of the measures $\mu_x$ is known in statistical physics as ``no linear response'' and can appear even 
in relatively simple dynamical systems~\cite{BerglundGentzKuehn,BerglundGentzKuehn1,50}.
\end{remark}
 
Our next natural goal is now to check under which abstract assumptions on the original ODE problems,
the condition (A4) (that is equation (\ref{eq: led to loo 7})) is satisfied.

\section{Convergence of the limiting generator $\mathcal{L}^{0,\delta}$}
\label{sec: convergence of Lod}

In this section we investigate requirements for condition (A4) to hold, which is the main 
assumption in Theorem \ref{thm: from Led to L00} and it is also our last missing piece for proving convergence 
of the first moments for the slow process for the coupled  deterministic systems (\ref{eq: fast slow ODE a and y}). 
Let us recall that the operator $\mathcal{L}^{0,\delta}$ is explicitly given by (\ref{eq: generator L0d})
 where the drift term $F^\delta$ and the diffusion matrix $A^\delta$ are explicitly given by 
(\ref{eq: F and A 1})
and by the alternative expressions in Lemma \ref{lem: alternative representations of the expressions}. 
These alternative expressions use the solution operator  $\phi_x^{\delta,t}$ solving equation (\ref{eq: phi_xi}). 
Thus, a first step towards proving (A4) is to understand the behavior of $\phi_x^{\delta,t}$ 
in the limit $\delta \to 0$: 
 
\begin{lem} \textbf{(Behavior of the solution operator as $\delta \to 0$)}
\label{lem: cont dependence of the solution oper} \\
	Under the previous assumptions, the following statements are true:\\
	\textbf{(i)} For every $T>0$ and $\omega \in \Lambda$, there exists a positive constant 
	$\beta(T ,\omega) > 0$ (which is independent of $x,y$ and $\delta$) such that: 
	\begin{equation}
	| \phi_{x}^{\delta,t}(y) - \phi_{x}^{0,t} (y) |_\infty \leq \sqrt{\delta} \beta(T ,\omega),
	\end{equation}
	where $|\cdot|_\infty$ denotes the supremum norm in $\mathbb{R}^m$. This implies that for 
	all $\omega \in \Lambda$ we have
	\begin{equation}
	\phi_{x}^{\delta,t}(y) \to \phi_{x}^{0,t} (y) \text{\quad as $\delta \to 0$ uniformly in $x, y$ 
	and $t \in [0,T]$ }.
\end{equation}
Furthermore, it holds that 
	\begin{equation}
	\label{eq: lem cont dependence of the solution oper 1}
	\mathbb{E}[| (\phi_{x}^{\delta,t}(y)) - \phi_{x}^{0,t} (y) |_\infty ] \leq \sqrt{\delta} \beta(T),
	\end{equation}
	where $\beta(T):= \mathbb{E}\left[\beta(T ,\omega)\right] < \infty$ \\
	\textbf{(ii)} There exists a version of the stochastic process $\phi_x^{\delta,t}(y)$ such 
	that for a.a.~$\omega \in \Lambda$ the map $x \mapsto \phi_x^{\delta,t}(y)$ is continuously 
	differentiable for every $t$ and the gradient $\nabla_{x}\phi_x^{\delta,t}(y)$ satisfies 
	the linear ODE
  \begin{equation}\label{eq: lem cont dependence of the solution oper 2}
	\frac{\rmd}{\rmd t} \nabla_{x}\phi_x^{\delta,t}(y) = \nabla_x g(x,\phi_x^{\delta,t}(y)) + 
	\nabla_y g(x,\phi_x^{\delta,t}(y))\nabla_{x}\phi_x^{\delta,t}(y) \quad \nabla_x\phi_x^{\delta,0}(y) = 0.
	\end{equation}
	Furthermore, for a.a. $\omega \in \Lambda$ we have
	\begin{equation}
	\nabla_x \phi_x^{\delta,t}(y) \to \nabla_x \phi_x^{0,t}(y) \text{\quad as $\delta \to 0$ 
	uniformly in $x,y$ and $t \in [0,T]$}.
	\end{equation}
\end{lem}

\begin{proof}
\textbf{(i)} Due to the definition of the solution operator, it follows immediately that 
for any $t \in [0,T]$
	\begin{align*}
	| \phi_{x}^{\delta,t} (y) - \phi_{x}^{0,t} (y)|_\infty &\leq \int_0^t 
	|g(x,\phi_{x}^{\delta,t} (y) ) - g(x,\phi_{x}^{0,t} (y) )|_\infty ~\rmd s + \sqrt{\delta} 
	|V(t) (\omega)|_\infty\\
	&\leq C(x) \int_0^t | \phi_{x}^{\delta,s} (y) - \phi_{x}^{0,s} (y)|_\infty ~\rmd s + 
	\sqrt{\delta}|V(t)(\omega)|_\infty \\
	&\leq \tilde{C} \int_0^t | \phi_{x}^{\delta,s} (y) - \phi_{x}^{0,s} (y)|_\infty~\rmd s + 
	\sqrt{\delta} \underbrace{\sup_{t \in [0,T]}|V(t)(\omega)|_\infty}_{=:\alpha(T,\omega)},
	\end{align*}
	where $\tilde{C} := \sup_{x \in \mathbb{R}^d} C(x) <\infty$ due to the boundedness of 
	$\nabla_xg$. Due to Gronwall's lemma it follows that for all $t\in [0,T]$
	\begin{equation}
	| \phi_{x}^{\delta,t} (y) - \phi_{x}^{0,t} (y)|_\infty \leq \sqrt{\delta} 
	\alpha(T ,\omega) \exp(CT) \leq \sqrt{\delta} \beta(T ,\omega),
	\end{equation}
	where we have set $\beta(T,\eta):= \alpha(T ,\eta) \exp(CT)$.
	Further we see that 
	\begin{equation*}
	\mathbb{E} [\beta(T ,\cdot)]= \txte^{CT} \mathbb{E}[\alpha(T ,\cdot) ]< \infty,
	\end{equation*}
	which implies, by monotonicity of the integral, equation (\ref{eq: lem cont dependence of the solution oper 1}).
	
	\textbf{(ii)} For the pathwise differentiability of the process $\phi_x^{\delta,t}$ 
	see Lemma \ref{lem: differentiability of the solution oper wrt to x} (or \cite[Theorem 4.2]{52}).
	Due to \textbf{(i)} we see further that for a.a. $\omega \in \Lambda$ 
	\begin{gather*}
	\nabla_x g(x,\phi_x^{\delta,t}(y)) \to \nabla_x g(x,\phi_x^{0,t}(y)), \\
	\nabla_y g(x,\phi_x^{\delta,t}(y)) \to \nabla_y g(x,\phi_x^{0,t}(y)) \text{\quad as 
	$\delta \to 0$ uniformly in $x,y$ and $t \in [0,T]$.}
	\end{gather*}
	Hence, the last equation is a consequence of continuous dependence of ODEs on the coefficients.
\end{proof}

After having understood the behavior of $\phi_x^{\delta, t}$ in the limit $\delta \to 0$ we 
now want to come back to the generator $\mathcal{L}^{0,\delta}$ given in (\ref{eq: generator L0d}). 
Its coefficients, which use  the solution operator $\phi_x^{\delta,t}$, are given in 
Lemma \ref{lem: alternative representations of the expressions}. 
Seeing these expressions and Lemma \ref{lem: cont dependence of the solution oper} one might be 
tempted to conclude the convergence of $F^\delta, A^\delta$ and as a consequence 
equation (\ref{eq: led to loo 7}). Unfortunately, it is not that simple, because for general 
functions $g$ the expressions $F_0^0, F_1^0$ and $A_0^0$ in 
Lemma \ref{lem: alternative representations of the expressions} will not be well-defined. In fact, 
they are only then well-defined, when the flow $\phi_x^{0,t}(y)$ has strong mixing properties. 
These considerations motivate the following definitions:

\begin{defn} \textbf{(Decay of correlations for deterministic systems)} 
\label{defin: doc} \\
We say that the flow $\phi_x^{0,t}(y)$ is \emph{mixing} with \emph{decay of correlations} 
$C(t;x)$ provided that there exists an $\alpha>0$ such that
	for all continuous functions $v,w: \mathbb{T}^m \to \mathbb{R}$, lying in the H\"{o}lder space $(C^{0, \alpha}, \parallel \cdot \parallel_\alpha)$, we have
	\begin{gather*}
	\Big|\int_{\mathbb{T}^m} v(z) w(\phi_x^{0,t}(z)) d\mu_x(z) - \int_{\mathbb{T}^m} 
	v(z)~\rmd\mu_x(z) \int_{\mathbb{T}^m} w(z) d\mu_x(z)\Big| \leq  
	C(t;x) \parallel v \parallel_\alpha\parallel w\parallel_\alpha,\\ 
	\text{ with }   C(t;x) \to 0 \text{\quad  as $t \to \infty$ for all $x \in \mathbb{R}^d$ } .
	\end{gather*}
	We say that the decay of correlations is \emph{summable}
	provided that 
	\begin{equation*}
	\int_0^\infty C(t;x) ~\rmd t < \infty \text{\quad for all } x \in \mathbb{R}^d,
	\end{equation*}
	and we say that the decay of correlations is \emph{exponential} provided that for every 
	$x \in \mathbb{R}^d$ there exist constants $C(x),\rho(x) > 0$ such that
	\begin{equation*}
	C(t;x) = C(x) \txte^{- \rho(x) t}.
	\end{equation*}
\end{defn}

\begin{remark} 
\label{rem: summable dec of corel implies exist of the expr}
Note that in the special case where either $\int_{\mathbb{T}^m} v(z) ~\rmd\mu_x(z)=0$ 
or $\int_{\mathbb{T}^m} w(z) ~\rmd\mu_x(z)=0$ holds, summable decay of correlations implies that
	\begin{equation*}
	\int_0^\infty\Big|\int_{\mathbb{T}^m} v(z) w(\phi_x^{0,t}(z)) ~\rmd\mu_x(z)\Big| ~\rmd t < \infty.
	\end{equation*}	
\end{remark}

\begin{lem} \textbf{(Decay of correlations for stochastic systems)} 
\label{lem: doc for stochastic systems} \\
	Fix a $\delta > 0$.
	For all continuous functions $v,w: \mathbb{T}^m \to \mathbb{R}$ we have
	\begin{equation*}
	\Big|\int_{\mathbb{T}^m} v(z) \mathbb{E} [ w(\phi_x^{\delta,t}(z)) ]
	~\rmd\mu_x^\delta(z) - \int_{\mathbb{T}^m} v(z) ~\rmd\mu_x^\delta(z) 
	\int_{\mathbb{T}^m} w(z) ~\rmd\mu_x^\delta(z) \Big| \leq \tilde{C}(\delta;x)
	\parallel v \parallel_\infty \parallel w \parallel_\infty \txte^{-\rho(\delta;x)t}.
	\end{equation*}
	In particular, this implies that the stochastic flow has exponential decay of correlations in the sense 
	of Definition~\ref{defin: doc}.
\end{lem}

\begin{proof}
	This is an easy application of \cite[Theorem 6.16]{PavliotisStuart}: 
	\begin{align*}
	\Big|	\int_{\mathbb{T}^m} v(z) \mathbb{E} [w(\phi_x^{\delta,t}(z)) ]~\rmd\mu_x^\delta(z) 
	&- \int_{\mathbb{T}^m} v(z) ~\rmd\mu_x(z) \int_{\mathbb{T}^m} w(z) ~\rmd\mu_x^\delta(z) \Big| \\
	&= \Big|\int_{\mathbb{T}^m} v(z) \Big\{\mathbb{E} [w(\phi_x^{\delta,t}(z))] -
	\int_{\mathbb{T}^m} w(\tilde{z}) ~\rmd\mu_x^\delta(\tilde{z})  \Big\}~\rmd\mu_x^\delta(z)\Big| \\
	&\leq \Big|\int_{\mathbb{T}^m} v(z) \tilde{C}(\delta;x) \parallel w \parallel_\infty 
	\txte^{-\rho(\delta;x) t }~\rmd\mu_x^\delta(z)\Big| \\
	&\leq \tilde{C}(\delta;x)\parallel v \parallel_\infty \parallel w \parallel_\infty 
	\txte^{-\rho(\delta;x)t}.
	\end{align*}
	This finishes the proof.
\end{proof}

\begin{defn} 
\textbf{(Stochastically stable decay of correlations)} 
\label{defin: stochastic stability} \\
	Let  $v,w: \mathbb{T}^m \to \mathbb{R}$. 
	Assume that the deterministic flow $\phi_x^{0,t}$ has decay of correlation $C(t;x)$.
	We say that $\phi_x^{0,t}$ has \emph{stochastically stable decay of correlations} 
	provided that for all small enough $\delta > 0$ and $x \in \mathbb{R}^d$ 
	\begin{equation} \label{eq:stdoc}
	\tilde{C}(\delta;x) \txte^{-\rho(\delta;x)t}  \leq C(t;x),
	\end{equation}
	where the constants on the left side are as in Lemma~\ref{lem: doc for stochastic systems}.
\end{defn}

\begin{remark} \label{rem: stoch_stable_doc}
Stochastically stable decay of correlations means that the mixing behavior of the system is not slowed down under stochastic perturbations. Due to Lemma \ref{lem: doc for stochastic systems}, it 
is natural to assume this form of stochastic stability for a large class of flows $\phi_x^{0,t}$ with decay of correlations, since the noise itself provides exponential decay of correlations. Recall that many of our examples concern flows with summable but subexponential decay of correlations $C(t;x)$. 
Note that in this case, we can fix an arbitrarily small $\delta_0 > 0$ and then obtain inequality~\eqref{eq:stdoc} for any $\delta \geq \delta_0$ by adapting the constants $\tilde C(\delta,x)$ accordingly, due to the exponential term. At the same time, we know that the inequality holds for $\delta=0$ by definition of $C(t;x)$. Hence, except for some presumably non-generic discontinuity when $\delta \to 0$, we may conclude the stochastically stable decay of correlations in these cases.
\end{remark}

These notions allow to prove the following statement concerning $F_0^0, F_1^0$ and $A_0^0$:

\begin{lem} 
\label{lem: convergence of the expressions}
Assume that the unperturbed flow $\phi_x^{0,t}$ has 
summable decay of correlations $C(t;x)$ and stochastically stable decay of correlations 
in the sense of Definitions  \ref{defin: doc} and \ref{defin: stochastic stability}, and
that the centering condition  (\ref{eq: centering condition with delta})
is satisfied. Furthermore, consider, for $\delta \geq 0$, the well-defined expressions 
$F_1^\delta(x)$ \eqref{eq: altern repres 3}, $A_0^\delta(x)$ \eqref{eq: altern repres 4} 
and, for $g=g(y)$,
\begin{equation}\label{eq: lem conv of expres 4}
	F_0^\delta(x) =  \int_0^\infty \lim_{T \to \infty} \frac{1}{T} \int_0^T \mathbb{E} 
	[\nabla_x b\Big(x,\phi^{\delta,t}(\phi^{\delta,s}(y)(\omega)) \Big) ]
	b\Big(x,\phi^{\delta,s}(y)(\omega)\Big) \rmd s ~\rmd t,
\end{equation}
which hold for all $y\in \mathbb{T}^m$ and a.a.~$\omega \in \Lambda$ by ergodicity 
(cf.~Lemma \ref{lem: alternative representations of the expressions}).

Then we have 
	\begin{equation}
	F_1^\delta \to F_1^0, \text{\quad}A_0^\delta \to A_0^0 \text{\quad as $\delta \to 0$ uniformly in }x,
	\end{equation}
	and, 
	in the case that $g=g(y)$, we additionally obtain
	\begin{equation}
	F_0^\delta \to F_0^0 \text{\quad as $\delta \to 0$ uniformly in } x.
	\end{equation}
\end{lem}

\begin{proof}	
	We first want to ensure that all considered expressions (\ref{eq: altern repres 3}), 
	(\ref{eq: altern repres 4}) and (\ref{eq: lem conv of expres 4}) are well-defined 
	for all $\delta \geq 0$. For (\ref{eq: altern repres 3}) this is trivial. 
	For (\ref{eq: altern repres 4}) note that for a.a.~$\omega\in \Lambda$, due to the 
	centering condition (\ref{eq: centering condition with delta}), Lemma \ref{lem: doc for stochastic systems} 
	and the stochastic stability we have componentwise in the tensor product
	\begin{align*}
	\Big|\lim_{T \to \infty} \frac{1}{T} \int_0^T b\Big(x, \phi_x^{\delta,s}(y)(\omega) \Big)
	\otimes \mathbb{E} [b \Big(x,\phi_x^{\delta,t}(\phi_x^{\delta,s}(y)(\omega)) \Big) ]~\rmd s \Big| 
	&= \Big|\int_{\mathbb{T}^m} b(x,y) \otimes 
	\mathbb{E} [ b(x,\phi_x^{\delta,t}(y) ] ~\rmd\mu_x^\delta(y)\Big| \\
	&\leq C_1(b) C(t;x) 
	\end{align*}
	($C_1(b)$ is a constant which depends on $b$)
	and analogously for (\ref{eq: lem conv of expres 4}) in the case that $g=g(y)$.

	We now start by estimating the difference $F_1^\delta - F_1^0$ for $\delta > 0$.
	Let $\epsilon > 0$ and define, for $T> 0$, $F_1^{\delta, T}:= \frac{1}{T} \int_0^Ta(x, \phi_x^{\delta,s})~\rmd s$. 
	For any $\delta > 0$ we have that 
	\begin{align*}
	|F_1^\delta - F_1^0| \leq |F_1^\delta - F_1^{\delta,T}| + |F_1^{\delta,T} - F_1^{0,T}| + |F_1^{0,T}- F_1^0|.
	\end{align*}
	For each $\delta > 0$ we can fix a $T = T_0$, which is independent of $\delta$ and $x,y,\omega$, 
	such that the first and last difference become smaller that $\frac{\epsilon}{3}$.
	To see this, note that the sequence $\frac{1}{T} \int_0^T \sup_{\delta,x,y,\omega} 
	|a\Big(x, \phi_x^{\delta,s}(y)(\omega) \Big)|\rmd s$ is bounded from above and increasing, 
	hence it converges.
	Moreover, due to Lemma \ref{lem: cont dependence of the solution oper} and due to the Lipschitz 
	continuity of the vector field  $a$, we have that
\begin{equation}
	|F_1^{\delta,T_0} - F_1^{0,T_0}| = \frac{1}{T_0} \int_0^{T_0} |a(x, \phi_x^{\delta,s}(y)) - 
	a(x, \phi_x^{0,s}(y))|~\rmd s \leq \sqrt{\delta} C(T_0,\omega) \to 0 \text{ for $\delta \to 0$}.
\end{equation}
Hence, for a.a. $\omega$ we have
\begin{equation}
F_1^\delta \to F_1^0 \text{\quad as $ \delta \to 0 $ uniformly in $x,y$.}
\end{equation}
Next, for estimating $A_0^\delta - A_0^0$ we we define 
	\begin{align*}
	a_\delta^{i,j}(t; x,y,\omega) &:= \lim_{T \to \infty}  \frac{1}{T} \int_0^T b^i
	\Big(x, \phi_x^{\delta,s}(y) \Big) \mathbb{E}\Big[ b^j\Big(x,\phi_x^{\delta,t}
	(\phi_x^{\delta,s}(y))  \Big) \Big] ~\rmd s, \\
	a_0^{i,j}(t; x,y,\omega) &:=  \lim_{T \to \infty} \frac{1}{T} 
	\int_0^T b^i\Big(x, \phi_x^{0,s}(y) \Big) b^j\Big(x,\phi_x^{0,t+s}(y) \Big)~\rmd s, \\
	a_\delta^{i,j,T}(t; x,y,\omega) &:= 	\frac{1}{T} \int_0^T b^i\Big(x, \phi_x^{\delta,s}(y) 
	\Big) \mathbb{E}\Big[ b^j \Big(x,\phi_x^{\delta,t}(\phi_x^{\delta,s}(y) \Big) \Big] ~\rmd s, \\
	a_0^{i,j,T}(t; x,y,\omega) &:= \frac{1}{T} \int_0^T b^i\Big(x, \phi_x^{0,s}(y)\Big) 
	b^j\Big(x,\phi_x^{0,t+s}(y) \Big)~\rmd s.
	\end{align*}
	As before we split
	\begin{align*}
	|a_\delta^{i,j} - a_0^{i,j}| \leq |a_\delta^{i,j} - a_\delta^{i,j,T}| + |a_\delta^{i,j,T} 
	- a_0^{i,j,T}| + |a_0^{i,j,T}- a_0^{i,j}|.
	\end{align*}
	The sequence 
	\begin{equation}\label{eq: lem convergence of the expressions, eq1} 
	\frac{1}{T} \int_0^T \sup_{\delta,x,y,\omega} \Big| b^i(x, \phi_x^{\delta,s}(y)) 
	\mathbb{E} \Big[ b^j\Big(x,\phi_x^{\delta,t}(\phi_x^{\delta,s}(y)) \Big) \Big] \Big|~ \rmd s
	\end{equation}
	is bounded from above and increasing, hence it converges for every $t$.
	Hence, we can find a $T = T_0(t)$, which is independent of $\delta$ and and $x,y$ and 
	$\omega$ such that the first and last terms of equation (\ref{eq: lem convergence of the expressions, eq1}) 
	become smaller than $\epsilon$. With this $T_0$ we have
	\begin{align*}
	|	a_\delta^{i,j,T_0}(t; x,y,\omega) &- a_0^{i,j,T_0}(t; x,y,\omega)| \\
	&\leq   \frac{1}{T_0} \int_0^{T_0}| b^i\Big(x, \phi_x^{\delta,s}(y) \Big)\mathbb{E}
	\Big[ b^j\Big(x,\phi_x^{\delta,t}(\phi_x^{\delta,s}(y))  \Big) \Big] - 
	b^i\Big(x, \phi_x^{0,s}(y)\Big) b^j\Big(x,\phi_x^{0,t+s}(y)\Big)|~\rmd s \\
	&\leq   \frac{1}{T_0} \int_0^{T_0} \underbrace{\Big|b^i\Big(x, \phi_x^{\delta,s}(y)\Big)\Big|}_{\leq C_1 }
	\underbrace{\Big| \mathbb{E} \Big[ b^j\Big(x,\phi_x^{\delta,t}(\phi_x^{\delta,s}(y))\Big) - 
	b^j\Big(x,\phi_x^{0,t+s}(y)\Big) \Big] \Big|}_{\leq \sqrt{\delta}C_2(t) \text{ due to 
	Lemma~\ref{lem: cont dependence of the solution oper}}} |~\rmd s  \\
	&+
	\frac{1}{T_0} \int_0^{T_0} \Bigg|\underbrace{b^j\Big(x,\phi_x^{0,t+s}(y) \Big)}_{\leq C_1} 
	\underbrace{\Bigg\{b^i\Big(x, \phi_x^{\delta,s}(y) \Big)- b^i\Big(x,\phi_x^{0,s}(y) \Big) 
	\Bigg\}}_{\leq \sqrt{\delta} C_3(T_0,\omega) \text{ due to Lemma \ref{lem: cont dependence of the solution oper}}} 
	\Bigg|~\rmd s \\
	&\leq \sqrt{\delta} C_4(t,T_0,\omega) \to 0 \text{\quad for }  \delta \to 0,
	\end{align*}
	where $ C_1, C_2, C_3, C_4$ denote positive constants.
	Hence, for all $t$ and $\omega$ we have
	\begin{equation}
	\label{eq: lemma convergence of the expressions, eq2}
	a_\delta^{i,j}(t; x,y,\omega) \to a_0^{i,j}(t; x,y,\omega) \text{\quad as } 
	\delta \to 0 \text{, uniformly in } x,y.
	\end{equation}
	Due to the assumption on the fast dynamics we know further that 
	for any fixed $t,x,y,\omega$ we have 
	\begin{equation}
	\label{eq: lemma convergence of the expressions, eq4}
	|a_\delta^{i,j}(t; x,y,\omega)| < C(t;x) \text{\quad  for } \delta \text{ sufficiently small.}
	\end{equation} 
	Using (\ref{eq: lemma convergence of the expressions, eq2}) 
	and (\ref{eq: lemma convergence of the expressions, eq4})
	we get by the dominated convergence theorem 
	\begin{equation}\label{eq: lemma convergence of the expressions, eq5}
	\int_0^\infty a_\delta^{i,j}(t; x,y,\omega) ~\rmd t \to \int_0^\infty a_0^{i,j}(t; x, y) 
	~\rmd t \text{ \quad as } \delta \to 0.
	\end{equation}
	Due to equation (\ref{eq: lemma convergence of the expressions, eq2}) the convergence is 
	uniform in $x\in \mathbb{R}^d$, $y\in \mathbb{T}^m$.
	From (\ref{eq: lemma convergence of the expressions, eq5}), it follows that  
	\begin{equation}
	A_0^\delta \to A_0^0 \text{\quad as $\delta \to 0$ uniformly in }x \in \mathbb{R}^d.
	\end{equation}
	Finally, we deal with the difference 
	$|F_0^\delta - F_0^0|$ in case that $g$ is independent of $x$.
	Proceeding as in our previous computations we can verify that 
	\begin{align*}
	\lim_{T \to \infty} \frac{1}{T} \int_0^T \mathbb{E}  \Big[ \nabla_x b\Big(x,&\phi_x^{\delta,t}
	(\phi_x^{\delta,s}(y)) \Big) \Big] b\Big(x,\phi_x^{\delta,s}(y) \Big)~\rmd s \\ &\to 
	\lim_{T \to \infty} \frac{1}{T} \int_0^T \nabla_x b\Big(x,\phi_x^{0,t+s}(y)\Big) 
	b\Big(x,\phi_x^{0,s}(y)\Big)\rmd s \text{ \quad as } \delta \to 0,
	\end{align*}
	uniformly in $x,y$ and for $t \in [0,T]$. This implies, due to 
	the stochastically stable decay of correlations of $\phi$ that 
	\begin{equation}
	F_0^\delta \to F_0^0 \text{\quad as $\delta \to 0$ uniformly in } x.
	\end{equation}	
	This finishes the proof.	
\end{proof}

It remains to deal with the term $F_0^0$ in case $g$ does also depend on $x$. The crucial ingredients 
are equations~\eqref{eq: lem convergence of F_0 in the case gxy 2} 
and~\eqref{eq: lem convergence of F_0 in the case gxy 1} such that we can formulate the following result:

\begin{lem} 
\label{lem: convergence of F_0 in the case gxy}
For the case that $g=g(x,y)$ also depends on $x$, we assume that the unperturbed flow $\phi_x^{0,t}$ has summable and stochastically 
stable decay of correlations  wrt. an ergodic invariant measure $\mu_x^0$ on $\mathbb{T}^m$. Additionally, we assume that the centering 
condition (\ref{eq: lem convergence of F_0 in the case gxy 2}) and, for any 
$y \in \mathbb{T}^m$, the growth condition (\ref{eq: lem convergence of F_0 in the case gxy 1}) 
are satisfied.

Then we obtain:
\begin{enumerate}
\item Setting 
	\begin{align*}
	f_0^\delta(t,x) := \lim_{T \to \infty} \frac{1}{T} \int_0^T \mathbb{E} 
	\Big[ \nabla_x b\Big(x &,\phi_x^{\delta,t}(\phi_x^{\delta,s}(y)) \Big) \\ &+ 
	\nabla_yb\Big(x,\phi_x^{\delta,t}(\phi_x^{\delta,s}(y))\Big)\nabla_x
	\phi_x^{\delta,t}(\phi_x^{\delta,s}(y))  \Big] b\Big(x,\phi_x^{\delta,s}(y) \Big) ~\rmd s,
	\end{align*}
	we have that
	\begin{equation}
	\parallel f_0^0(t,\cdot)\parallel_\infty \leq h(t), \text{\quad for a  function $h$ 
	with} \int_0^\infty h(t) ~\rmd t < \infty.
	\end{equation}
	\item For $\delta \geq 0$ small enough, $h(t)$ is an upper bound for $f_0^\delta$, 
the expression	
	\begin{equation*}
	F_0^\delta(x) =  \int_0^\infty 	f_0^\delta(t,x)  ~\rmd t
	\end{equation*}
	is well-defined 
	and we have
	\begin{equation}\label{lem: Fd to F0 conclusion}
	F_0^\delta \to F_0^0 \text{\quad as $\delta \to 0$ uniformly in } x \in \mathbb{R}^d .
	\end{equation}
	\end{enumerate}
\end{lem}

\begin{proof}
We must first ensure that all expressions $F_0^\delta$ are well-defined. It is easy to see 
that for all $\delta \geq 0$ we have
\begin{equation}
	\Big|\lim_{T \to \infty} \frac{1}{T} \int_0^T \mathbb{E} \Big[ \nabla_x 
	b\Big(x,\phi_x^{\delta,t}(\phi_x^{\delta,s}(y)) \Big) \Big] b\Big(x,\phi_x^{\delta,s} 
	(y) \Big) ~\rmd s \Big|_\infty \leq C_2 C(t;x),
	\end{equation}
for a constant $C_2>0$. 	
Secondly for $\delta = 0 $, we set 
$w^{x}:= \nabla_yb(x,y)$
 and $v^{t,x}:=
  \nabla_x \phi_x^{0,t}(y) b(x,y)$ in the definition of decay of 
correlations and, using condition (\ref{eq: lem convergence of F_0 in the case gxy 2}), 
 we observe that 
	\begin{align*}
		&\Big|\lim_{T \to \infty} \frac{1}{T} \int_0^T \mathbb{E} \Big[ \nabla_yb
		\Big(x,\phi_x^{\delta,t}(\phi_x^{\delta,s}(y))\Big)\nabla_x\phi_x^{\delta,t}
		(\phi_x^{\delta,s}(y)) \Big] b\Big(x,\phi_x^{\delta,s}(y)\Big)~ \rmd s\Big| \\
		&\leq C(t,x) \parallel w^x \parallel_\alpha \parallel v^{t,x} \parallel_\alpha.
	\end{align*}
	This fact together with the growth assumption (\ref{eq: lem convergence of F_0 in the case gxy 1}) yields
	\begin{equation*}
	\parallel f_0^0(t, \cdot) \parallel_\infty \leq \sup_{x \in \mathbb{R}^d}  \Big\{ C(t;x) (  \parallel w^x \parallel_\alpha \parallel\nabla_x \phi_x^{0,t}(\cdot) b(x,\cdot)\parallel_\alpha + C_2)  \Big\} 
  =: h(t), \text{\quad } \int_0^\infty h(t) ~\rmd t < \infty,
	\end{equation*}  
	which, in particular, implies  that $F_0^0$ is well-defined.
	Furthermore, due to stochastically stable decay of correlations, proceeding as in Lemma \ref{lem: convergence of the expressions} (and using also Lemma 
	\ref{lem: cont dependence of the solution oper} \textbf{(ii)}) we can show that 
	\begin{equation*}
	f_0^\delta \to f_0^0, \text{\quad } \parallel f_0^\delta(t, \cdot) \parallel_\infty\leq h(t).
	\end{equation*}
Finally, we can conclude~\eqref{lem: Fd to F0 conclusion} by dominated convergence. 
\end{proof}

This allows us now to conclude the main result of this section, Theorem \ref{thm B}.

\begin{proof}[Proof of Theorem \ref{thm B}]
The statement follows immediately from Lemmas \ref{lem: convergence of the expressions} 
and \ref{lem: convergence of F_0 in the case gxy}.	
\end{proof}

\begin{remark}
(i) Condition (\ref{eq: lem convergence of F_0 in the case gxy 1}) seems to be a relatively 
strong mixing condition, which may be difficult to verify for certain practical examples.
Indeed, one observes that $\nabla_x \phi_x^{\delta,t} (y)$ solves the first order linear 
inhomogeneous ODE (\ref{eq: lem cont dependence of the solution oper 2}). Thus, 
$\nabla_x \phi_x^{\delta,t} (y) $  can be calculated by variation of constants and is 
explicitly given by the formula
	\begin{equation*}
	\nabla_x \phi_x^{\delta,t} (y) = \txte^{ \int_0^t \nabla_yg(x,\phi_x^{\delta,\tau}(y))~\rmd\tau} \Bigg(  
	\int_0^t  \txte^{- \int_0^s \nabla_y g(x, \phi_x^{\delta,\tau}(y)) ~\rmd\tau}\nabla_x 
	g(x,\phi_x^{\delta,s}(y)) ~\rmd s + y \Bigg).
	\end{equation*}
Assuming for simplicity that the matrices $\txte^{ \int_0^t \nabla_yg(x,\phi_x^{\delta,\tau}(y))~\rmd\tau}$ 
and $\txte^{- \int_0^s \nabla_y g(x, \phi_x^{\delta,\tau}(y))~\rmd\tau}$ commute, we obtain
from the last equation
	\begin{align*}
	|\nabla_x \phi_x^{\delta,t} (y)|_\infty &\leq
	\parallel \nabla_x g \parallel_\infty \int_0^t \txte^{\parallel \nabla_yg \parallel_\infty 
	(t-s)}~\rmd s + \txte^{\parallel \nabla_y g\parallel_\infty t}.
	\end{align*}	
	From this we conclude that
	\begin{equation*}
	\sup_{x,y,\omega,\delta}|\nabla_x \phi_x^{\delta,t} (y)|_\infty \leq K 
	\txte^{\parallel \nabla_y g\parallel_\infty t},
	\end{equation*}
	where the constant 
	$$K:= 	\parallel \nabla_x g \parallel_\infty \int_0^\infty \txte^{-\parallel \nabla_yg 
	\parallel _\infty s}~\rmd s + 1$$
	 is independent of $t$.	
	Thus, the growth condition (\ref{eq: lem convergence of F_0 in the case gxy 1}) might hold if the 
	unperturbed flow $\phi_x^{0,t}$ has exponential decay of correlations $C(t;x) \leq  C \txte^{-\rho t}$, 
	for all $x \in \mathbb{R}^d$, with	$\rho \geq \parallel \nabla_yg \parallel_\infty$. This inequality describes precisely the boundary of what we might optimistically expect as possible decay rates for correlations and a further investigation is left as an open problem here.\\
		(ii) The centering condition (\ref{eq: centering condition with delta}) might seem a strong assumption 
	at first glance because it must be satisfied for all $\delta >0$ and $x$.  
	However, the parameter $\delta > 0$ has the effect of only ``streching'' the invariant density 
	$\rho_\infty^\delta(y;x)$,
	so that the function $b$ has to be simply some function which is in accordance with the symmetry 
	of the invariant densities. The condition can also be relaxed by allowing the operator $\mathcal{L}_2$ 
	to be perturbed as well. More precisely, assume that the function $b$ satisfies 
	\begin{equation*}
	\int_{\mathbb{T}^m} b(x,y)~\rmd\mu_x^0(y) = 0, \text{\quad for all } x \in \mathbb{R}^d.
	\end{equation*}
	We consider suitable perturbed vector fields $b^\delta$ satisfying the centering condition 
	(\ref{eq: centering condition in theorem}), for which additionally we have
	\begin{equation*}
	b^\delta \to b \text{\quad uniformly.}
	\end{equation*}
	For example, we can consider functions of the form 
	\begin{equation*}
	b^\delta(x,y):= b(x,y) - \int_{\mathbb{T}^m} b(x,z) \rho_\infty^\delta(z;x) dz
	\end{equation*}
	We then define the perturbed operators
	\begin{equation*}
	\mathcal{L}_2^\delta u := b^\delta \cdot \nabla_x u,
	\end{equation*}
	\begin{equation*}
	\mathcal{L}^{\epsilon,\delta} =\frac{1}{\epsilon^2} \mathcal{L}_{1}^\delta + 
	\frac{1}{\epsilon} \mathcal{L}_2^\delta + \mathcal{L}_3
	\end{equation*}
	and 
	\begin{equation*}
	\mathcal{L}^{0,\delta} f :=  (-\mathcal{P}^\delta \mathcal{L}_2^\delta 
	[\mathcal{L}_1^{\delta}]^{-1} \mathcal{L}_2^\delta \mathcal{P}^\delta 
	+ \mathcal{P}^\delta \mathcal{L}_3^\delta \mathcal{P}^\delta )f
	\end{equation*}
	and we can repeat the proof of Theorem \ref{thm: from Led to L00} to get the statement. 
\end{remark}
%
%

\section{Weakly-coupled systems}
\label{sec: weakly coup sys}

\subsection{Main result}
\label{sub: main results for weakly coupled systems}

To provide an intermediate alternative to the strong mixing assumption (see 
condition (\ref{eq: lem convergence of F_0 in the case gxy 1})), we are also consider a simpler case of 
so-called \emph{weakly-coupled} systems. These are systems with coupling occurring only in lower times 
scales and they are given by equation (\ref{eq: coupl in slower time scales 0DE 1 }). 
We also consider the corresponding stochastic version 
\begin{equation}\label{eq: coupl in slower time scales SDE 1 }
\begin{split}
\frac{\rmd x_\epsilon}{\rmd t} &= a(x_\epsilon, y_\epsilon) + \frac{1}{\epsilon} 
b(x_\epsilon, y_\epsilon),\text{\quad} x_\epsilon(0) = x_0,  \\
\frac{\rmd y_\epsilon}{\rmd t} &= \frac{1}{\epsilon^2} g( y_\epsilon) + \frac{1}{\epsilon} 
\left(h(x_\epsilon,y_\epsilon) + \sqrt{\delta} \frac{\rmd V}{\rmd t}\right) + r(x_\epsilon,y_\epsilon),
\text{\quad} y_\epsilon(0) = y_0.
\end{split}
\end{equation}
We are going to use now the assumptions (A1)-(A2), (A4)-(A5), and suitable centering
an correlation decay conditions but not (A6) to finally be able to prove Theorem~\ref{thm C}.
For any $\delta > 0$ we set 
\begin{align*}
\tilde{\mathcal{L}}_1^\delta &:= g(y) \cdot \nabla_y + \frac{1}{2}\delta I: \nabla_y\nabla_y , \\
\tilde{\mathcal{L}}_2 &:= b(x,y) \cdot \nabla_x + h(x,y) \cdot \nabla_y = 
\mathcal{L}_2^c + \mathcal{L}_2^{nc}, \\
\tilde{\mathcal{L}}_3&=  a(x,y) \cdot \nabla_x + r(x,y)\cdot \nabla_y,
\end{align*}
with the commutative part
$\mathcal{L}_2^c := b(x,y) \cdot \nabla_x$ and the remainder 
$\mathcal{L}_2^{nc} := h(x,y) \cdot \nabla_y$. 
The operator 
\begin{equation*}
\tilde{\mathcal{L}}^{\epsilon, \delta} = \frac{1}{\epsilon^2}\tilde{\mathcal{L}}_1^\delta 
+ \frac{1}{\epsilon}\tilde{\mathcal{L}}_2 + \tilde{\mathcal{L}}_3
\end{equation*}
is the backward Kolmogorov operator associated with the SDE (\ref{eq: coupl in slower time scales SDE 1 }). 
Assume that the centering condition (\ref{eq: centering condition with delta}) is satisfied.
Consider the perturbation expansion 
\begin{equation} \label{eq: perturbation ansatz for wc}
u^{\epsilon, \delta} = u_0^\delta + \epsilon u_1^\delta + \epsilon^2 u_2^\delta + \cdots
\end{equation}
which we substitute into the backward Kolmogorov equation 
 \begin{equation}\label{eq: bachward kolm eq with eps wc}
\frac{\rmd u^{\epsilon,\delta}}{\rmd t} = \tilde{\mathcal{L}}^{\epsilon,\delta}u^{\epsilon,\delta} 
:= \Big(\frac{1}{\epsilon^2}\tilde{\mathcal{L}}_1^\delta + \frac{1}{\epsilon}\tilde{\mathcal{L}}_2 
+ \tilde{\mathcal{L}}_3 \Big)u^{\epsilon,\delta}.
\end{equation}
Via the perturbation analysis given in Section \ref{sec: perturbation anal for wcsys app} of the 
Appendix, we arrive at the following equation for the leading order $u_0^\delta$
 \begin{equation}\label{eq: eq for uo wc}
 \frac{\rmd u_0^\delta}{\rmd t} = \tilde{F}^{\delta} \cdot \nabla_x u_0^\delta 
+ \frac{1}{2} A^\delta(x)A^\delta(x)^\top : \nabla_x\nabla_x u_0^\delta.
 \end{equation}
Here the drift coefficient in the homogenized equation  
(\ref{eq: apr eq for barrho fp 2 }) now changes to
\begin{equation}
\label{eq: coupl in slower time scales Fdelta }
\tilde{F}^\delta(x):= \int_{\mathbb{T}^m} \Big(a(x,y) + \nabla_x \Phi^\delta(y;x) b(x,y) 
+ \nabla_y\Phi^\delta(y;x)h(x,y) \Big) \rho_\infty^\delta(y;x) ~\rmd y
\end{equation}
and the diffusion coefficient $A^\delta(x)$ remains unchanged 
\begin{equation}
\begin{split}
A^\delta(x)A^\delta(x)^\top  &=  \frac{1}{2} \Big( A_0^\delta(x) + A_0^\delta(x)^\top  \Big), \\
A_0^\delta(x) &= 2 \int_{\mathbb{T}^m} \Big(b(x,y) \otimes \Phi^\delta(y;x)\Big) 
\rho_\infty^\delta(y;x) ~\rmd y.
\end{split}
\end{equation}
Note that (see for example \cite[Result 11.8]{PavliotisStuart}) the solution $\Phi^\delta$ of 
the cell problem admits the representation formula 
\begin{equation*}
\Phi^\delta (y;x) = \int_0^\infty \mathbb{E} \Big[ b(x, \phi^{\delta,t}(y)) \Big] ~\rmd t,
\end{equation*}
where the stochastic process $\phi^{\delta,t}(y)$ satisfies equation (\ref{eq: solution oper phi without x})
and the term $\mathbb{E} [ b(x, \phi^{\delta,t}(y)) ]$ decays exponentially fast as $t\to \infty$ 
(see \cite[Theorem 6.16]{PavliotisStuart}). The above considerations allow us to repeat the arguments from 
the previous sections and we get following theorem.

\begin{thm} \textbf{(Convergence of the slow process for weakly-coupled systems)}
\label{thm: coupled systems} \\
Assume (A1)-(A2) and that the unperturbed flow $\phi^{0,t}$ has summable stochastically stable decay 
of correlations $C(t)$ in the sense of Definitions~\ref{defin: doc} and \ref{defin: stochastic stability}. Furthermore, 
assume that the centering condition (\ref{eq: centering condition with delta})
is satisfied and define the operator $\tilde{\mathcal{L}}^{0,\delta}$ on $C^2_\txtc(\mathbb{R}^d)$ by 
\begin{equation}
\tilde{\mathcal{L}}^{0,\delta}u := \tilde{F}^\delta \cdot \nabla_xu + \frac{1}{2}A^\delta(x)A^\delta(x)^\top 
:\nabla_x\nabla_xu.
\end{equation}
In the case that $h$ does not vanish everywhere, we assume additionally that the centering 
condition~\eqref{eq: lem convergence of F_0 in the case gxy 2} and the 
growth condition \eqref{eq: growth cond in coupl sys thm} hold.
	
Then following statements are true: 
\begin{enumerate}
	\item[\textbf{(i)}] There exist vector fields $\tilde{F}^0(x)$ and $A^0(x)$ such that 
	\begin{equation}
	\tilde{F}^\delta \to \tilde{F}^0, \text{\quad}A^\delta \to A^0, \text{\quad uniformly in $x$ as $\delta \to 0$}, 
	\end{equation}
	where  $A^0$ is explicitly given by (\ref{eq: thm weakly coupled systems A0}) 
	and the vector field $\tilde{F}^0$ is given by (\ref{eq: thm weakly coupled systems F0}).
	\item[\textbf{(ii)}]
	For every $f \in C^2_\txtc(\mathbb{R}^d)$  
	\begin{equation}\label{eq: led to loo weakly coup} 
	\lim_{\delta \to 0}\tilde{\mathcal{L}}^{0, \delta} f = \tilde{\mathcal{L}}^{0,0} f \text{\quad uniformly},
	\end{equation}
	where the operator $\tilde{\mathcal{L}}^{0,0}$ is defined by 
	\begin{equation}
	\tilde{\mathcal{L}}^{0,0} u := \tilde{F}^0 \cdot \nabla_xu + \frac{1}{2}A^0(x)A^0(x)^\top :\nabla_x\nabla_xu,
	\end{equation}
	and $\bar{\tilde{\mathcal{L}}}^{0,0}$ generates the strongly continuous semigroup $T(t)^{0,0}$ on $X$. 
	\item[\textbf{(iii)}] Let $T^{\epsilon,\delta}$ be the semigroup on $\hat{C}(\mathbb{R}^d\times \mathbb{T}^m)$ 
	generated by $\bar{\mathcal{L}}^{\epsilon,\delta}$.
	Then for every $f \in C_0(\mathbb{R}^d)$ and every sequence $\{\epsilon_k\}_{k\geq 0}$ with $\epsilon_k \to 0$ 
	for $k \to \infty$, there exists a subsequence $ \{ \epsilon_{k_m} \}_{m \geq 0}$ such that 
	\begin{equation}
	\label{eq: led to loo couple sys 2 } 
	\lim_{m \to \infty} \sup_{0 \leq t \leq \hat{T}} \parallel T^{\epsilon_{k_m},0}(t)f  - 
	T^{0,0}(t)f \parallel_\infty = 0.
	\end{equation}
	\item[\textbf{(iv)}]
	For $\epsilon > 0$ let $(X^\epsilon(t;\xi ,\eta), Y^\epsilon(t;\xi,\eta))$ be the solution of the 
	ODE (\ref{eq: coupl in slower time scales 0DE 1 }). 
	Then for every initial condition $f \in \hat{C}(\mathbb{R}^d)$ and every sequence $\{\epsilon_k\}_{k\geq 0}$ 
	with $\epsilon_k \to 0$ for $k \to \infty$, there exists a subsequence $ \{ \epsilon_{k_m} \}_{m \geq 0}$ such 
	that 
	\begin{equation*}
f(X^{\epsilon_{k_m}}(t; \xi , \eta)) \to T^{0,0}(t) f(\xi), 
\text{\quad uniformly in $\xi \in \mathbb{R}^d$, $\eta \in \Omega$ 
	and $t\in [0,\hat{T}]$}. 
\end{equation*}
	\end{enumerate}
\end{thm} 

\begin{proof}
	The arguments needed for the proof are identical with those given in 
	Sections \ref{sec: coupled systems} and \ref{sec: convergence of Lod}. Thus we 
	omit their exact repetition. We only want to note that in the case that $h \equiv 0$ the 
	term $\nabla_y\Phi^\delta(x,y)h(x,y)$ in (\ref{eq: coupl in slower time scales Fdelta }) 
	vanishes, so that we can repeat the arguments from Lemma \ref{lem: convergence of the expressions} 
	to get the first statement. In the general case that $h$ does not vanish everywhere, the 
	term $\nabla_y\Phi^\delta(x,y)h(x,y)$ in equation (\ref{eq: coupl in slower time scales Fdelta }) 
	cannot be neglected. Thus we need to pose the additional assumptions (\ref{eq: lem convergence 
	of F_0 in the case gxy 2}) and  (\ref{eq: growth cond in coupl sys thm}) (which ensure especially that the expression
\begin{equation*}
\int_0^\infty \int_{\mathbb{T}^m} \mathbb{E} \Big[\nabla_yb(x, \phi_x^{\delta,t}(y)) \nabla_y \phi_x^{\delta,t}(y) \Big] h(x,y) \rho^\delta_\infty(y;x) dy dt <\infty
\end{equation*} is well-defined) and then we proceed 
	as in Lemma \ref{lem: convergence of F_0 in the case gxy} to get the first statement also for 
	this case. Finally we note that for the second statement we repeat the arguments from 
	Theorem \ref{thm B}, for the third statement we need to repeat the proof of 
	Theorem \ref{thm: from Led to L00} and for the last statement see the proof of 
	Corollary \ref{cor: after thm Led to Loo}.
\end{proof}
As we can see from the formulation of Theorem~\eqref{thm: coupled systems}, we do not have to assume any additional growth condition for $\phi^{0,t}$ in case $h$ in~\eqref{eq: coupl in slower time scales SDE 1 } vanishes. If $h \neq 0$, the assumed growth condition~\eqref{eq: growth cond in coupl sys thm} for the weakly-coupled system is clearly weaker than growth condition~\eqref{eq: lem convergence of F_0 in the case gxy 1} for the more general case: in~\eqref{eq: lem convergence of F_0 in the case gxy 1}, the integrability has to hold uniformly over all $x \in  \R^d$, whereas $\phi^{0,t}$ does not depend on $x$ in the weakly-coupled situation, hence the simplification to~\eqref{eq: growth cond in coupl sys thm}.

\subsection{Numerical example}
\label{sec: numerical example}

As an application of the previous Section~\ref{sub: main results for weakly coupled systems}, we 
consider a weakly-coupled system on $\mathbb{R}\times \mathbb{R}^3$ with chaotic fast dynamics 
on the Lorenz attractor. Let us recall that the classical Lorenz equations are given by the 
three-dimensional ODE system
\begin{equation}\label{eq: lorenz eq 1}
\begin{split}
\frac{\rmd y_1}{\rmd t} &= s(y_2-y_1),  \\
\frac{\rmd y_2}{\rmd t} &= \rho y_1 - y_2 -y_1y_3,   \\
\frac{\rmd y_3}{\rmd t} &= y_1y_2 - \beta y_3,
\end{split}
\end{equation}
with the parameters $s, \rho , \beta >0$, where, in particular, $s$ is called the \emph{Prandtl 
number} and $\rho$ is called the \emph{Rayleigh number}. For the standard values $s = 10, 
\rho = 28, \beta = 8/3$, the equations are ergodic with invariant measure $\mu$ supported on the 
Lorenz attractor $\Omega$. We now consider, motivated by \cite[Section 11.7.2]{PavliotisStuart} 
and \cite[Section 6.4]{35}, the following weakly-coupled systems on $\mathbb{R}\times\mathbb{R}^3$:
\begin{equation}
\label{eq: model problem 1 lorenz 2 wco}
\begin{split}
\frac{\rmd X^{\epsilon,\delta}}{\rmd t} &= -X^{\epsilon,\delta} + \frac{1}{\epsilon}\frac{4}{90} 
Y^{\epsilon,\delta}_2 \\
\frac{\rmd Y^{\epsilon,\delta}_1}{\rmd t} &= \frac{10 }{\epsilon^2}(Y^{\epsilon,\delta}_2-
Y^{\epsilon,\delta}_1) + X^{\epsilon,\delta}Y_3^{\epsilon,\delta} +\delta \frac{\rmd U}{\rmd t}\\
\frac{\rmd Y^{\epsilon,\delta}_2}{\rmd t} &= \frac{1}{\epsilon^2} (28Y^{\epsilon,\delta}_1 - 
Y^{\epsilon,\delta}_2 -Y^{\epsilon,\delta}_1Y^{\epsilon,\delta}_3)  - X^{\epsilon,\delta}+
\delta \frac{\rmd V}{\rmd t} \\
\frac{\rmd Y^{\epsilon,\delta}_3}{\rmd t} &=\frac{1}{\epsilon^2}( Y^{\epsilon,\delta}_1
Y^{\epsilon,\delta}_2 - \frac{8}{3} Y^{\epsilon,\delta}_3) + X^{\epsilon,\delta} 
Y_1^{\epsilon,\delta} Y_2^{\epsilon,\delta} +\delta \frac{\rmd W}{\rmd t}.
\end{split}
\end{equation}
In Figure \ref{fig: 8} sample paths of the process $X^{\epsilon,\delta}$ 
solving (\ref{eq: model problem 1 lorenz 2 wco}) for different values of $\epsilon$ and 
$\delta$ are shown. These paths illustrate that the deterministic flow displays 
stochastic-looking/chaotic oscillations but one does really need to look at the limiting 
behaviour as $\epsilon\ra 0$ to fail to see the visual difference between a deterministic
and a stochastic process.

\begin{figure}[htbp]
	\centering
	\begin{overpic}[width=0.45\textwidth]{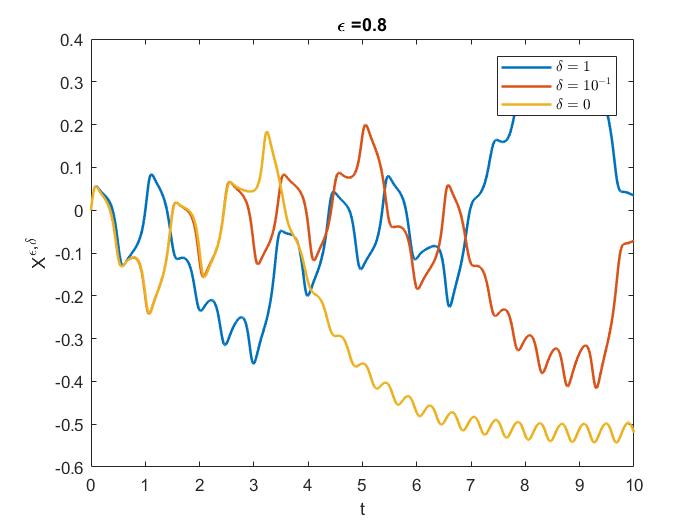}
	\end{overpic}
	\begin{overpic}[width=0.45\textwidth]{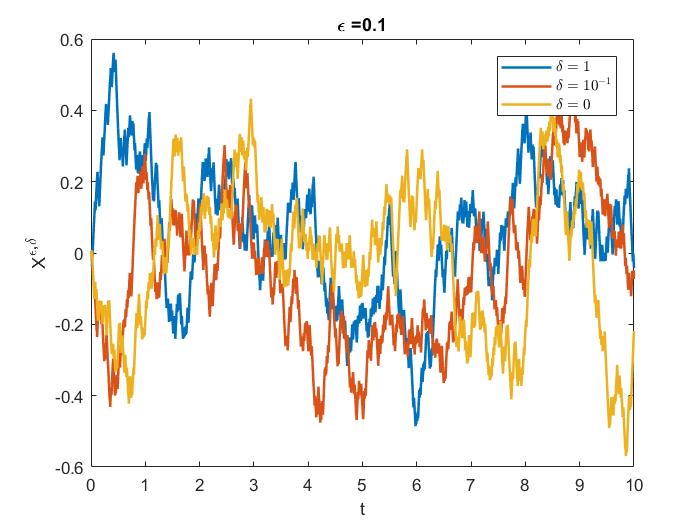}
\end{overpic}
	\caption{\label{fig: 8}Sample paths of the process $X^{\epsilon, \delta}$ satisfying 
	equation (\ref{eq: model problem 1 lorenz 2 wco}), with the initial 
	condition $[x^0, y_1^0, y_2^0, y_3^0]^\top =[0,13.93,20.06,26.87]^\top$, for different 
	values of $\epsilon$ and $\delta$.}
\end{figure}

The fast subsystem has the ergodic measure $\mu$ supported on the Lorenz attractor $\Omega$.
Let $Q \subset \mathbb{R}^3$ be a sufficiently large cube containing $ \Omega$. By identifying 
the opposite sides of the cube and rescaling the coordinates we can assume, without loss of 
generality, that $Q = \mathbb{T}^3$ is the torus, so that the theory from the previous sections 
can be applied. We note further that it has been already verified numerically in \cite{35} that 
the $y_2$ coordinate has zero average with respect to $\mu$ and as a consequence that the centering 
condition (\ref{eq: centering cond for d=0}) is satisfied.
Theorem \ref{thm: coupled systems} states that for every $f \in C_0(\mathbb{R})$ and every 
sequence $\{\epsilon_k\}_{k\geq 0}$  with $\epsilon_k \to 0$ for $k \to \infty$ there exists a 
subsequence $ \{ \epsilon_{k_m} \}_{m \geq 0}$ such that
\begin{equation} \label{eq:  example weakly coupled systems convergence of Xe0}
\mathbb{E}^{\mu} [f(X^{\epsilon_{k_m},0}(t))]\to \mathbb{E}[f(X(t))]
\text{\quad as $m \to \infty$ uniformly in $t \in [0,\hat{T}]$},
\end{equation}
where the process $X$ solves the SDE
\begin{equation}\label{eq:  example weakly coupled systems SDE for X}
\frac{\rmd X}{\rmd t} = - X + \sigma \frac{\rmd W}{\rmd t}, \text{\quad } X(0) = \xi.
\end{equation}
Note that equation (\ref{eq:  example weakly coupled systems SDE for X}) describes an 
Ornstein-Uhlenbeck process  which has the unique solution given by
\begin{equation*}
X_t = \txte^{-t} \xi + \sigma \txte^{-t} \int_0^t \txte^{ \tau} ~\rmd W_\tau. 
\end{equation*}
In general we know that for a square integrable function $f$ on $[0,T]$, the random variable 
$\int_0^T f(t) ~\rmd W_t$ is normally distributed with variance $\int_0^T f(t)^2 ~\rmd t$ and 
from this fact it is easy to see that $X_t$ is normally distributed with
\begin{equation*}
X_t \sim N(\txte^{- t} \xi, \frac{\sigma^2}{2} \txte^{-2 t}(\txte^{2t} - 1)).
\end{equation*}
The exact value of $\sigma$ is given by formula (\ref{eq: thm weakly coupled systems F0}). In the 
following we use the estimate $\sigma^2 \simeq 0.126$ calculated in \cite{35}.

Furthermore, since $C_0(\mathbb{R}) \subset C_b(\mathbb{R})$, equation 
(\ref{eq:  example weakly coupled systems convergence of Xe0}) is slightly 
weaker than uniform convergence in distribution of the 
process $X^{\epsilon_{k_m},0}(t)$ towards $X(t)$. 
The following Figures \ref{fig: 9} and \ref{fig: 10} verify 
equation (\ref{eq:  example weakly coupled systems convergence of Xe0}) numerically.

\begin{figure}[htbp]
	\centering
	\begin{overpic}[width=0.65\textwidth]{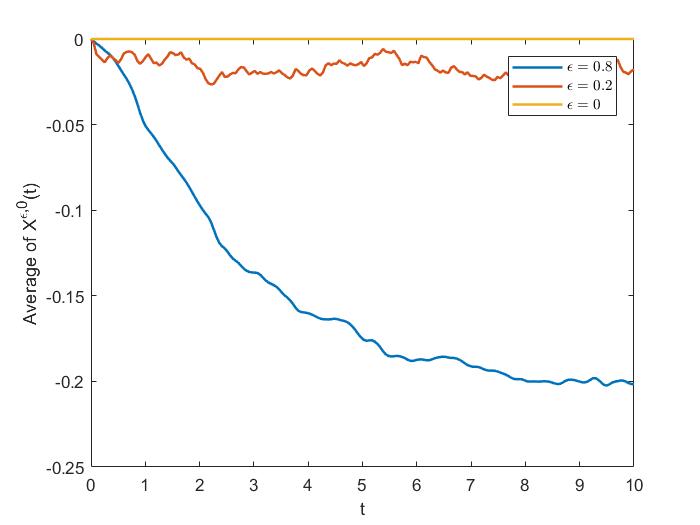}
	\end{overpic}
	\caption{\label{fig: 9}Average of the process $X^{\epsilon, 0}(t)$ satisfying equation 
	(\ref{eq: model problem 1 lorenz 2 wco}) for $\epsilon = 0.8$ and $\epsilon =0.2$ and 
	theoretical average for $\epsilon = 0$, i.e. for the limiting process $X(t)$ 
	satisfying (\ref{eq:  example weakly coupled systems SDE for X}) with the initial 
	condition $\xi=0$. The averages are taken over 1000 different realizations on the Lorenz 
	attractor. We observe that the average starts to converge towards the theoretical value but
	one really has to go to small $\epsilon$ to see the effect.}
\end{figure}

\begin{figure}[htbp]
	\centering
	\begin{overpic}[width=0.65\textwidth]{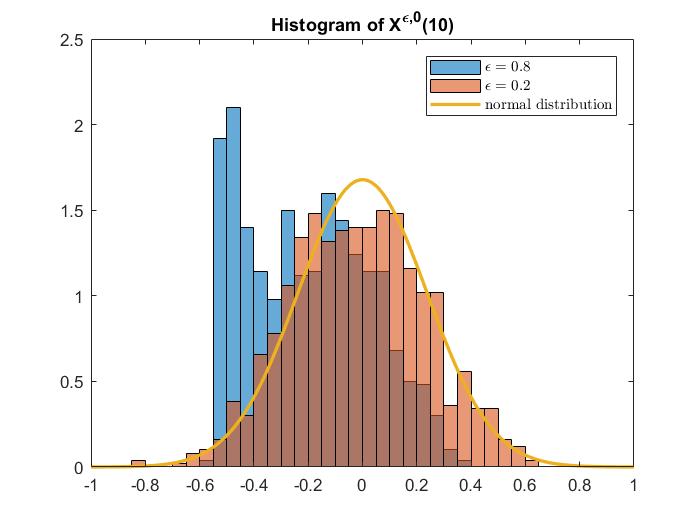}
	\end{overpic}
	\caption{\label{fig: 10}Histogram of the process $X^{\epsilon, 0}(10)$ satisfying 
	equation (\ref{eq: model problem 1 lorenz 2 wco}) for $\epsilon = 0.8$ and 
	$\epsilon =0.2$ and distribution of the limiting process $X(10)$, 
	solving (\ref{eq:  example weakly coupled systems SDE for X}) with the initial condition $\xi =0$.
	We used ensembles of 1000 realizations.}
\end{figure}

Figure~\ref{fig: 10} shows that we indeed obtain convergence of the deterministic fast-slow 
ODE to its limiting reduced flow SDE. We visualized this by capturing the same stationary 
distribution with both equations. This is the practical reduction effect one is looking for
since now the chaotic fast degrees of freedom are encoded in a low-dimensional SDE.

\section{Conclusion and outlook}
\label{sec: conclusion and outlook}

In this paper we have extended results on deterministic homogenization of 
fast-slow ODEs to the case where 
coupling of the fast and slow variables is part of the model. Our main strategy was to add small stochastic noise 
to the fast subsystem and then take two independent limits --- namely the zero-noise limit and 
the limit $\epsilon \to 0$ ---, which enabled us to use results and functional-analytical methods 
from stochastic systems. For generally coupled systems, we have succeeded to prove a certain 
weak form of convergence of the slow process, similarly to uniform convergence of the first moments, 
requiring strong mixing assumptions on the fast flow. However, for the intermediate case of 
weakly-coupled systems, the mixing assumptions are relatively mild. Our method also directly yields explicit 
expressions for the drift and diffusion coefficients of the limiting SDE. 

This paper can be seen as one of the first steps to understand homogenization of coupled fast-slow systems in continuous time and 
leaves open several relevant questions for further research. One task is to find, numerically and/or 
analytically, more direct examples from applications for which the strong mixing condition 
\eqref{eq: growth cond in coupl sys thm} is satisfied. Another goal will be to find alternative 
representations of the drift and diffusion coefficients of the limiting diffusion, such that 
potentially weaker or even no mixing assumptions are required, as seen in \cite{13,12}. In addition to that, it will be crucial to 
study the behavior of the higher moments of the slow process in order to prove weak convergence of the respective measures in $C([0,T],\mathbb{R}^d)$. 

\medskip

\textbf{Acknowledgements:} M.E.~and C.K.~gratefully acknowledge support by the DFG via the SFB TR 109
Discretization in Geometry and Dynamics. M.E. has also been supported by Germany's Excellence Strategy -- The Berlin Mathematics Research Center MATH+ (EXC-2046/1, project ID: 390685689). C.K.~acknowledges partial support by a Lichtenberg Professorship of 
the VolkswagenFoundation and partial support via the TiPES project funded by the European Unions Horizon 2020 research and innovation programme under grant agreement No. 820970. M.G. and C.K. acknowledge support
via the TUM International Graduate School of Science and Engineering via the project SEND.

\bibliographystyle{plain}
\bibliography{mybib}

\newpage
\appendix

\section{Convergence of the semigroup $T^{\epsilon,\delta}$ as $\epsilon \to 0$}

Let $X$ be a Banach space. 

\begin{defn}
\label{defin: ergodic contraction semigroup} 
Let $\{S(t)\}_{t \geq 0}$ be a strongly continuous semigroup on 
$X$ with infinitesimal generator $L$. $\{S(t)\}_{t \geq 0}$ is called an 
\emph{ergodic semigroup} if 
	\begin{equation} 
	\label{eq: ergodic contra semigroup property}
	\lim_{\lambda \to 0} \lambda \int_0^\infty \txte^{- \lambda t} 
	S(t) f~\rmd t = Pf \text{ \quad exists for all } f \in X.
	\end{equation}
We call $P$ the projection corresponding to the semigroup.
\end{defn}

\begin{remark}
\label{rem: suffic condition for ergodic contract semigrou}
A sufficient condition for (\ref{eq: ergodic contra semigroup property}) to 
hold is that $\lim_{t \to \infty} S(t) f$ exists for every $f\in X$ and then 
we also have that 
\begin{equation}
	Pf = \lim_{t \to \infty} S(t)f, \text{\quad } f\in X.
\end{equation}
Using semigroup notation we can rewrite the last equation as
\begin{equation}
\txte^{Lt}\cdot \to P \cdot \text{\quad as } t \to \infty.
\end{equation}
See also \cite[Remark 7.5]{30}.
\end{remark}

\begin{lem} \label{lem: suitable banach space L}
For any fixed $\delta > 0$ consider the operators $\mathcal{L}^{\epsilon,\delta}$ 
and $\mathcal{L}_1^\delta$ defined as in (\ref{eq: bachward kolm eq with eps}) on 
$C^2_\txtc(\mathbb{R}^d \times \mathbb{T}^m)$. 
Let $X:= (C_0(\mathbb{R}^d \times \mathbb{T}^m), \parallel \cdot \parallel_\infty)$ be 
the Banach space of continuous functions, which vanish for $\parallel x \parallel_2\to \infty$. 
Then the following statements are true\\
	\textbf{(i)} $\mathcal{L}^{\epsilon,\delta}$ generates a strongly continuous contraction 
	semigroup  $(T^{\epsilon,\delta}(t))_{t \geq 0}$ on $X$. \\
	\textbf{(ii)} $\mathcal{L}_1^\delta$ generates an ergodic, strongly continuous contraction 
	semigroup $(S^\delta(t))_{t \geq 0}$ on $X$. 
\end{lem}

\begin{proof}
	(i) Let $\psi_t(x,y)$ denote the solution map of the SDE corresponding to the generator 
	$\mathcal{L}^{\epsilon,\delta}$. For $f\in X$ define 
	\begin{equation*}
	T^{\epsilon,\delta}(t) f(x,y) := \mathbb{E} [f(\psi_t(x,y)) ].
	\end{equation*} 
	Note that due to our smoothness assumptions on $a,b,g$, $\phi^\epsilon_t(x,y)$ is continuous 
	with respect to the initial condition $(x,y)$. 
	We know check that: \\
	\textbf{(i-a)} 
	\begin{equation*}
	T^{\epsilon,\delta}(t): X \to X.
	\end{equation*}
	To see this, we first note that if $(x,y) \to (x_0,y_0)$ in $\mathbb{R}^d \times \mathbb{T}^m$, 
	then $\psi_t(x,y) \to \psi_t(x_0,y_0) $ which implies due to the dominated convergence theorem, 
	using that $f$ is bounded, that
	\begin{equation*}
	|(T^{\epsilon,\delta}(t) f)(x,y) - (T^{\epsilon,\delta} f )(x_0,y_0)| = \Big|\mathbb{E} 
	\Big(f(\psi_t(x,y)) - f(\psi_t(x_0,y_0)) \Big) \Big| \to 0.
	\end{equation*}
	Hence, $T^{\epsilon,\delta}(t) f  \in C(\mathbb{R}^d \times \mathbb{T}^m)$. 
	Similarly, using that $\psi_0(x,y) = (x,y)$ it is easy to see that for every fixed $y \in \mathbb{T}^m$ 
	and $t \in \mathbb{R}_+$ we have that $\parallel x \parallel_2 \to \infty \Rightarrow \parallel \psi_t(x,y)
	\parallel_2 \to \infty \Rightarrow f(\psi_t(x,y)) \to 0,$ which implies by dominated convergence that 
	\begin{equation*}
	(T^{\epsilon,\delta} f)(x,y) = \mathbb{E}[ f(\psi^t(x,y)) ] \to 0 \text{\quad as } 
	\parallel x \parallel_2\to \infty.
	\end{equation*}
	Hence, $T^{\epsilon,\delta}(t)f \in X$. \\
	\textbf{(i-b)}
	\begin{equation*}
	T^{\epsilon,\delta}(t+s)f = T^{\epsilon,\delta}(t) T^{\epsilon,\delta}(s)f, \text{\quad } T^\epsilon(0) 
	= \textnormal{Id}.
	\end{equation*}
	This follows immediately from the semigroup property of the solution map $\psi_t$. \\
	\textbf{(i-c)} 
	\begin{equation*}
	\lim_{t \to 0^+} \parallel T^{\epsilon,\delta}(t) f - f \parallel_\infty = 0.
	\end{equation*}
	Assume first for simplicity that $f\in C^2_\txtc(\mathbb{R}^d \times \mathbb{T}^m)$.
	Due to the It\^{o} formula we have that
	\begin{equation*}
	f(\psi_t(x,y))  = f(x,y) + \int_0^t \mathcal{L}^\epsilon f(\psi_s(x,y)) ds + M_t,
	\end{equation*}
	where $M_t$ is a martingale (which implies that $\mathbb{E} [M_t] = 0$).
	Thus, taking expectations we have
	\begin{equation*}
	|(T^{\epsilon,\delta}(t) f)(x,y) - f(x,y)|\leq \mathbb{E} [\int_0^t| 
	\mathcal{L}^{\epsilon,\delta} f(\psi_s(x,y))| ~\rmd s ].
	\end{equation*}
	Note that there exists a constant $C^\epsilon$, which depends only on the coefficients 
	of $\mathcal{L}^{\epsilon,\delta}$ such that 
	\begin{equation*}
	||\mathcal{L}^{\epsilon,\delta} f ||_\infty \leq C^\epsilon \underbrace{( \parallel f 
	\parallel_\infty + \parallel \nabla f \parallel _\infty + \parallel \nabla^2 f 
	\parallel_\infty) }_{< \infty, \text{ since }  f\in C^2_\txtc(\mathbb{R}^d \times \mathbb{T}^m) }
	< \infty.
	\end{equation*}
	Hence 
	\begin{equation*}
	\parallel(T^{\epsilon,\delta}(t) f)(x,y) - f(x,y)\parallel_\infty \leq t C^\epsilon \to 0 \text{ as } t \to 0^+.
	\end{equation*}
	Last equation implies strong continuity in $C^2_\txtc(\mathbb{R}^d \times \mathbb{T}^m)$, thus by 
	density also in $C_0(\mathbb{R}^d \times \mathbb{T}^m)$. \\
	\textbf{(i-d)} 
	\begin{equation*}
	\parallel T^{\epsilon,\delta}(t) f \parallel_\infty \leq \parallel f \parallel_\infty.
	\end{equation*}
	This is easy to see. All in all, $\mathcal{L}^{\epsilon,\delta}$ generates a strongly continuous 
	contraction semigroup on $X$. \\
	\textbf{(ii)} Analogously we can show that 
	$\mathcal{L}_1^\delta$ generates a strongly continuous contraction semigroup $S^\delta(t)_{t \geq 0}$ on $X$. 
	For the ergodicity it suffices to show (see also \cite[Remark 7.5]{30})\\
	\textbf{(ii-e)} 
	\begin{equation}
	S^\delta(t) f \to P^\delta f  \text{\quad in $X$ as } t \to \infty,
	\end{equation}
	where $P^\delta$ is the projection given by \eqref{eq: Projection Pd}
	Let $\tilde{\psi}_t(x,y)$ denote the flow of the SDE corresponding to $\mathcal{L}_1^\delta$. 
	Observe that due to the structure of the generator, the flow has the form 
	\begin{equation*}
	\tilde{\psi}_t(x,y) = (x, \phi_x^{\delta,t}(y)),
	\end{equation*}
	where $\phi_x^{\delta,t}(y)$ solves (\ref{eq: phi_xi}).
	Due to \cite[Theorem 6.16]{PavliotisStuart} we have
	\begin{equation*}
	\parallel S^\delta(t) f(x,y) - P^\delta f(x,y) \parallel_\infty \leq C \parallel 
	f \parallel_\infty \txte^{-\lambda t} \to 0 \text{\quad as } t\to \infty,
	\end{equation*}
	since the constant $C$ can be chosen to be independent of $x,y$ (due to the uniform bounds 
	on the coefficients of the SDE). This proves the ergodicity of the semigroup $S^\delta(t)$ on $X$.
\end{proof}

\begin{thm}\cite[Chapter 12, Theorem 2.4]{30} \label{thm: application of kurtz} \\
Fix a $\delta>0$ and let $\mathcal{L}^{\epsilon,\delta}$ 
	be the the operators as in (\ref{eq: bachward kolm eq with eps}). Define $\mathcal{P}^\delta$ 
	by (\ref{eq: Projection Pd}) and assume that the centering 
	condition (\ref{eq: centering condition in theorem}) is satisfied for all $x \in \mathbb{R}^d$. 	
Furthermore let  $\Phi^\delta$ be  the solution of the cell problem (\ref{eq: cell problem}).
	Define 
	\begin{equation*}
	D:= C^2_\txtc(\mathbb{R}^d) \subset X.
	\end{equation*}
	For every $f \in D$ let $h\in X$  denote the unique solution of the Poisson equation
	\begin{equation} 
	\label{cor: from Kurtz, poisson eq 2}
	\mathcal{L}_1^\delta h = - \mathcal{L}_2 \mathcal{P}^\delta f, \text{\quad} 
	\int_{\mathbb{T}^m} h(x,y) \rho_\infty^\delta(y;x) ~\rmd y = 0,
	\end{equation}
	whose existence and uniqueness is guaranteed due to the centering condition and the Fredholm alternative 
	and let $\mathcal{L}^{0,\delta}$ be the operator defined on $D$ by 
   (\ref{eq: apr eq for barrho fp}).
	Assume that the closure $\bar{\mathcal{L}^{0,\delta}}$ generates a strongly continuous contraction 
	semigroup $\{T(t)^{0,\delta}\}_{t\geq 0}$ on $C_0(\mathbb{R}^d)$. 
	Then we have for every $f \in \bar{D}$ and finite times $\hat{T} < \infty$ 
	\begin{equation}
	\label{eq: cor from Kurtz, eq}
	\lim_{\epsilon \to 0} \sup_{0 \leq t \leq \hat{T}} \parallel T^{\epsilon,\delta}(t)
	f - T(t)^{0,\delta}f \parallel_\infty = 0.
	\end{equation}		
\end{thm}

\begin{proof}
	The proof is taken from \cite[Chapter 12, Theorem 2.4]{30} but is included for convenience.
	From Lemma \ref{lem: suitable banach space L} follows that $\bar{\mathcal{L}_1^\delta}$ generates 
	the ergodic strongly continuous contraction semigroup $\{S(t)^\delta\}_{t\geq0}$ on $X$ and 
	$\bar{\mathcal{L}^\epsilon}$ generates the strongly continuous contraction semigroup 
	$\{T^{\epsilon,\delta}(t)\}_{t\geq 0}$ on $X$.
	We define 
	\begin{gather*}
	\mathcal{D}(\mathcal{L}_1^\delta) := \{ f \in L:\forall x: f(x;\cdot) \in C^2(\mathbb{T}^m) \} \\
	\mathcal{D}(\mathcal{L}_2) := \{ f \in L:\forall y: f(\cdot;y) \in C^1_c(\mathbb{R}^d) \} \\
	\mathcal{D}(\mathcal{L}_3) := \{ f \in L: \forall y: f(\cdot;y) \in C^1_c(\mathbb{R}^d) \}.
	\end{gather*}
	We observe that 
	\begin{equation*}
	D \subset \mathcal{D}(\mathcal{L}_1^\delta) \cap \mathcal{D}(\mathcal{L}_2) \cap \mathcal{D}(\mathcal{L}_3).
	\end{equation*}	
	Define further
	\begin{equation*}
	\mathcal{D}(V):= \{\mathcal{L}_2f: f=f(x) \in C_\txtc^2(\mathbb{R}^d)\} 
	\end{equation*}
	and
	\begin{equation*}
	\mathcal{R}(V):=  \{ f \in C^{2,0}_c(\mathbb{R}^d\times \mathbb{T}^m): f(x,\cdot) 
	\in C^2(\mathbb{T}^m), \mathcal{L}_1^\delta f \in L \}
	\end{equation*} 
	and the operator $V: \mathcal{D}(V) \to \mathcal{R}(V)$, acting via 
	\begin{equation*}
	V(f):= \Phi^\delta(y;x) \cdot \nabla_x f(x).
	\end{equation*}
	Note that since $b$ and the coefficients of $\mathcal{L}_1$ are smooth and $\mathcal{L}_1$ 
	is uniformly elliptic, $\Phi$ is smooth in both arguments (See also  \cite[Lemma 17.2]{PavliotisStuart} 
	for a similar situation). Having this in mind, it is easy to check that $R(V) \subset
	\mathcal{D}(\mathcal{L}_1^\delta) \cap \mathcal{D}(\mathcal{L}_2) \cap \mathcal{D}(\mathcal{L}_3)$ and 
	recalling the definitions of $\Phi^\delta$ and $\mathcal{L}_1^\delta$ we also see that $h= V(f)$ solves the 
	Poisson equation
	\begin{equation*}
	\mathcal{L}_1^\delta V(f) = - \mathcal{L}_2 f = - \mathcal{L}_2 \mathcal{P}f, \text{\quad } 
	\int_{\mathbb{T}^m} h(x,y) \rho_\infty^\delta(y;x) ~\rmd y = 0.
	\end{equation*}
	Hence, 
	\begin{equation*}
	D \subset \{f \in \mathcal{D}(\mathcal{L}_1^\delta) \cap\mathcal{D}(\mathcal{L}_2) \cap 
	\mathcal{D}(\mathcal{L}_3): \exists h \in 	\mathcal{D}(\mathcal{L}_1^\delta) 
	\cap\mathcal{D}(\mathcal{L}_1^\delta) \cap \mathcal{D}(\mathcal{L}_1^\delta): \mathcal{L}_1^\delta h 
	= -\mathcal{L}_2f\}.
	\end{equation*}
	The claim follows now from 	\cite[Chapter 1, Corollary 7.8]{30}, setting $A:= \mathcal{L}_2$, 
	$\Pi:= \mathcal{L}_3$ and  $B:= \mathcal{L}_1^\delta$.
\end{proof}

\section{Perturbation analysis for weakly-coupled systems}
\label{sec: perturbation anal for wcsys app}

In the following we follow \cite{PavliotisStuart} and \cite{35}. We provide the perturbation
expansions here for completeness as they are the most convenient tool to formally derive the
correct limiting behavior. Substituting (\ref{eq: perturbation ansatz for wc}) into the backward 
Kolmogorov equation (\ref{eq: bachward kolm eq with eps wc}) and collecting terms of the same 
powers we obtain a sequence of problems:
\begin{align}
\mathcal{O}(\frac{1}{\epsilon^2}): & \text{\quad} \tilde{\mathcal{L}}_1^\delta u_0^\delta 
= 0 \label{eq: per 1 bk wc} \\
\mathcal{O}(\frac{1}{\epsilon}): & \text{\quad} \tilde{\mathcal{L}}_1^\delta u_1^\delta 
= - \mathcal{L}_2^c u_0^\delta \label{eq: per 2 bk wc}\\
\mathcal{O}(1): & \text{\quad} \tilde{\mathcal{L}}_1^\delta u_2^\delta 
= \frac{\rmd u_0^\delta}{\rmd t} - \tilde{\mathcal{L}}_2 u_1^\delta - 
\tilde{\mathcal{L}}_3 u_0^\delta, \label{eq: per 3 bk wc}\\
... \nonumber
\end{align}
From equation (\ref{eq: per 1 bk wc}) it follows, due to the ergodicity property (\ref{eq: erg ass 1}) 
for $\tilde{\mathcal{L}}_1^\delta$, that the solution $u_0^\delta$  does not depend on $y$, in other words 
it is of the form
\begin{equation*}
u_0^\delta(x,y,t) = u_0^\delta(x,t).
\end{equation*}
To solve the second equation 
note that the centering condition (\ref{eq: centering condition with delta}) implies that $\mathcal{L}_2^c 
u_0^\delta$ is orthogonal to the null space of  $\Big(\tilde{\mathcal{L}}_1^\delta\Big)^*$.
Thus, by the Fredholm alternative equation (\ref{eq: per 2 bk wc}) is solvable and the solution is unique 
up to a constant lying in the null space of $\tilde{\mathcal{L}}_1^\delta$. We fix this constant by requiring 
\begin{equation}\label{eq: center cond for u_1}
\int_{\mathbb{T}^m} u_1^\delta(x,y) \rho_\infty^\delta(y;x)~\rmd y = 0 \text{\quad for all } x \in \mathbb{R}^d.
\end{equation}
Thus we can write formally
\begin{equation}
u_1^\delta = - \Big(\tilde{\mathcal{L}}_1^\delta\Big)^{-1}\mathcal{L}_2^c  u_0^\delta(x,t).
\end{equation}
We continue with the last equation (\ref{eq: per 3 bk wc}). Solvability requires that the right side is orthogonal 
to the null space of $\mathcal{L}_1$ and this leads the following equation for $u_0^\delta(x,t)$:
\begin{align}\label{eq: apr eq for barrho fp wc}
\frac{\rmd u_0^\delta}{\rmd t}(x,t) &=  \int_{\mathbb{T}^m}\rho_\infty^\delta(y;x)
\tilde{\mathcal{L}}_3 u_0^\delta(x,t) ~\rmd y - \int_{\mathbb{T}^m}\rho_\infty^\delta(y;x)
\tilde{\mathcal{L}}_2\Big(\tilde{\mathcal{L}}_1^\delta\Big)^{-1}\tilde{\mathcal{L}}_2 u_0^\delta(x,t)~\rmd y \nonumber\\
&= \Big(\mathcal{P}^\delta\tilde{\mathcal{L}}_3 \mathcal{P}^\delta -
\mathcal{P}^\delta\tilde{\mathcal{L}}_2\Big(\tilde{\mathcal{L}}_1^\delta\Big)^{-1}\tilde{\mathcal{L}}_2 
\mathcal{P}^\delta\Big)u_0^\delta(x,t). 
\end{align}
In this way we obtained a closed equation for the dominant term $u_0^\delta$ but we still have to evaluate 
the operators involved in it. Recall that $\Phi^\delta$ denotes the solution of the cell 
problem (\ref{eq: cell problem}). Thus, coming back to equation (\ref{eq: per 2 bk wc}), we observe that $u_1$ 
must have due to (\ref{eq: center cond for u_1}) the form 
\begin{equation}
u_1^\delta(x,y,t) = \Phi^\delta(y;x) \cdot \nabla_x u_0^\delta(x,t).
\end{equation}
Hence,
\begin{equation*}
\tilde{\mathcal{L}}_2 u_1^\delta = \underbrace{b \otimes \Phi : \nabla_x 
\nabla_x u_0^\delta + (\nabla_x \Phi b) \cdot \nabla_x u_0^\delta }_{= \mathcal{L}_2^c u_1} +
   \underbrace{(\nabla_y \Phi h )\cdot\nabla_x u_0^\delta}_{= \mathcal{L}_2^{nc} u_1^\delta},
\end{equation*}
Equation (\ref{eq: apr eq for barrho fp wc}) can be now re-written as
\begin{align}
\frac{\rmd u_0^\delta}{\rmd t} &= \mathcal{P}^\delta\tilde{\mathcal{L}}_3 u_0^\delta
-\mathcal{P}^\delta\tilde{\mathcal{L}}_2\underbrace{\tilde{\mathcal{L}}_1^{-1}\tilde{\mathcal{L}}_2  
u_0^\delta}_{=- u_1^\delta} \nonumber\\
&= I_1 + I_2,
\end{align}
with
\begin{align*}
I_1 &= \int_{\mathbb{T}^m} a(x,y) \rho_\infty^\delta(y;x) ~\rmd y \cdot \nabla_x u_0^\delta(x,t) 
\end{align*}
and 
\begin{align*}
I_2 = \int_{\mathbb{T}^m} \rho_\infty^\delta(y;x) &\Big(b(x,y) \otimes \Phi^\delta(y;x): \nabla_x \nabla_x u_0^\delta(x,t) \Big) ~\rmd y \\
&+ \int_{\mathbb{T}^m} \rho_\infty^\delta(y;x)\Big(\nabla_x \Phi^\delta(y;x)b(x,y) \Big) \cdot \nabla_x u_0^\delta(x,t)~\rmd y \\
&+ \int_{\mathbb{T}^m} \rho_\infty^\delta(y;x)\Big(\nabla_x \Phi^\delta(y;x)h(x,y) \Big) \cdot \nabla_x u_0^\delta(x,t)~\rmd y.
\end{align*}
Putting everything together we get (\ref{eq: eq for uo wc}).
\end{document}